\documentclass[10pt, reqno, a4paper]{amsart}

\usepackage{amssymb, amsmath, amsthm}
\usepackage{enumitem}
\usepackage[all]{xy}
\usepackage{mathtools}
\usepackage{extarrows}
\usepackage[normalem]{ulem}
\usepackage{stackrel}
\usepackage[english]{babel}

\usepackage{xcolor}

\usepackage[colorlinks=true, pdfstartview=FitV, linkcolor=blue, citecolor=blue, urlcolor=blue]{hyperref}

\newtheorem{theorem}{Theorem}[section]
\newtheorem*{theorem*}{Theorem}

\newtheorem*{maintheorem*}{Main Theorem}
\newtheorem{lemma}[theorem]{Lemma}
\newtheorem{corollary}[theorem]{Corollary}
\newtheorem{proposition}[theorem]{Proposition}

\newtheorem*{question*}{Question}

\theoremstyle{remark}
\newtheorem{remark}[theorem]{Remark}
\newtheorem{example}[theorem]{Example}

\theoremstyle{definition}
\newtheorem{definition}[theorem]{Definition}

\setlist[1]{labelindent=\parindent, leftmargin=*}

\DeclareMathOperator{\GL}{GL}
\DeclareMathOperator{\PGL}{PGL}
\DeclareMathOperator{\SL}{SL}

\DeclareMathOperator{\Aff}{Aff}
\DeclareMathOperator{\T}{T}

\DeclareMathOperator{\D}{D}
\DeclareMathOperator{\pr}{pr}
\DeclareMathOperator{\Aut}{Aut}

\DeclareMathOperator{\Cent}{Cent}
\DeclareMathOperator{\Bir}{Bir}
\DeclareMathOperator{\id}{id}
\DeclareMathOperator{\Tr}{Tr}
\DeclareMathOperator{\Spec}{Spec}

\DeclareMathOperator{\im}{im}
\DeclareMathOperator{\aff}{aff}
\DeclareMathOperator{\ant}{ant}

\DeclareMathOperator{\dom}{dom}

\DeclareMathOperator{\car}{char}

\DeclareMathOperator{\len}{len}

\DeclareMathOperator{\rRepl}{rRepl}

\DeclareMathOperator{\Norm}{Norm}
\DeclareMathOperator{\Rat}{Rat}

\DeclareMathOperator{\lociso}{lociso}

\def\dashmapsto{\mapstochar\dashrightarrow}

\newcommand{\OO}{\mathcal{O}}
\newcommand{\GG}{\mathbb{G}}

\newcommand{\CC}{\mathbb{C}}
\newcommand{\QQ}{\mathbb{Q}}
\newcommand{\PP}{\mathbb{P}}

\newcommand{\NN}{\mathbb{N}}
\renewcommand{\AA}{\mathbb{A}}

\newcommand{\kk}{\textbf{k}}

\newcommand{\aquot}{/ \! \! /}

\renewcommand{\D}{\mathcal{D}}

\newcommand{\set}[2]{\left\{\,#1 \ | \ #2\,\right\}}
\newcommand{\Bigset}[2]{\left\{\,#1 \ \Big| \ #2\,\right\}}

\makeatletter
\newcommand*{\da@rightarrow}{\mathchar"0\hexnumber@\symAMSa 4B }
\newcommand*{\da@leftarrow}{\mathchar"0\hexnumber@\symAMSa 4C }
\newcommand*{\xdashrightarrow}[2][]{%
  \mathrel{%
    \mathpalette{\da@xarrow{#1}{#2}{}\da@rightarrow{\,}{}}{}%
  }%
}
\newcommand{\xdashleftarrow}[2][]{%
  \mathrel{%
    \mathpalette{\da@xarrow{#1}{#2}\da@leftarrow{}{}{\,}}{}%
  }%
}
\newcommand*{\da@xarrow}[7]{%
  \sbox0{$\ifx#7\scriptstyle\scriptscriptstyle\else\scriptstyle\fi#5#1#6\m@th$}%
  \sbox2{$\ifx#7\scriptstyle\scriptscriptstyle\else\scriptstyle\fi#5#2#6\m@th$}%
  \sbox4{$#7\dabar@\m@th$}%
  \dimen@=\wd0 %
  \ifdim\wd2 >\dimen@
    \dimen@=\wd2 %
  \fi
  \count@=2 %
  \def\da@bars{\dabar@\dabar@}%
  \@whiledim\count@\wd4<\dimen@\do{%
    \advance\count@\@ne
    \expandafter\def\expandafter\da@bars\expandafter{%
      \da@bars
      \dabar@ 
    }%
  }%
  \mathrel{#3}%
  \mathrel{%
    \mathop{\da@bars}\limits
    \ifx\\#1\\%
    \else
      _{\copy0}%
    \fi
    \ifx\\#2\\%
    \else
      ^{\copy2}%
    \fi
  }%
  \mathrel{#4}%
}
\makeatother

\makeatletter
\ifcsname phantomsection\endcsname
\newcommand*{\qrr@gobblenexttocentry}[5]{}
\else
\newcommand*{\qrr@gobblenexttocentry}[4]{}
\fi
\newcommand*{\addsubsection}{%
	\addtocontents{toc}{\protect\qrr@gobblenexttocentry}%
	\subsection}
\makeatother

\title[Group Theoretical Characterizations of Rationality]{
Group Theoretical Characterizations of Rationality}
\author[A. Regeta, C. Urech, \and I. van Santen]
{Andriy Regeta, Christian Urech, \and Immanuel van Santen}

\address{\noindent Institut f\"{u}r Mathematik, Friedrich-Schiller-Universit\"{a}t Jena, \newline
	\indent Ernst-Abbe-Platz 2, DE-07737, Germany}
\email{andriyregeta@gmail.com}

\address{Departement Mathematik, 
	ETH Zürich,\newline
	\indent Rämistrasse 101, CH-8092 Zürich, Switzerland}
\email{christian.urech@gmail.com}

\address{Departement Mathematik und Informatik, 
	Universit\"at Basel,\newline
	\indent Spiegelgasse 1, CH-4051 Basel, Switzerland}
\email{immanuel.van.santen@math.ch}

\newcounter{claim}[theorem]
\renewcommand{\theclaim}{\arabic{claim}}

\newenvironment{claim}{\vspace{1.75\medskipamount} 
\refstepcounter{claim}\par\noindent\textbf{Claim \theclaim.}}{}

\begin{document}
\setcounter{tocdepth}{1}


\begin{abstract}
	Let $X$ be an irreducible variety and $\Bir(X)$ its group of birational transformations. We show that the group structure of $\Bir(X)$ determines whether $X$ is rational and whether $X$ is ruled.
	
	Additionally, we prove that any Borel subgroup of $\Bir(X)$ has derived length at most twice the dimension of $X$, with equality occurring if and only if $X$ is rational and the Borel subgroup is standard. We also provide examples of
	non-standard Borel subgroups of $\Bir(\PP^n)$ and $\Aut(\AA^n)$,
	thereby resolving conjectures by Popov and Furter-Poloni.
\end{abstract}

\maketitle

\tableofcontents

\section{Introduction}

We work over the field $\CC$ of complex numbers. An  important birational invariant of a variety $X$ is its group of birational transformations $\Bir(X)$. It is a natural question to ask whether the abstract group structure of $\Bir(X)$ characterizes rationality. The first main result of this paper answers this question positively:

\begin{theorem}
	\label{thm:charmain}
	Let $X$ be an irreducible variety such that $\Bir(X)$ is isomorphic as a group to $\Bir(\PP^n)$. Then $X$ is birationally equivalent to $\PP^n$. 
\end{theorem}

Under the additional assumption that $\dim(X)\leq n$, this result has already been shown in \cite{Ca2014Morphisms-between-, CaXi2018Algebraic-actions-}. In \cite{CaReXi2023Families-of-commut} an analogous result is proved for the automorphism group of affine varieties: if $\Aut(X)$ is isomorphic to $\Aut(\AA^n)$ for an affine variety $X$, then $X$ is isomorphic to $\AA^n$
(see also \cite{ReSa2024Maximal-commutativ, LiReUr2023Characterization-o, KrReSa2021Is-the-affine-spac, kraft2017automorphism} for variations of this problem). 
Similar results are known to hold for diffeomorphism groups of smooth manifolds as well as homeomorphism groups of certain topological manifolds \cite{Fi1982Isomorphisms-betwe, whittaker1963isomorphic}. Let us note that for $n=2$ the group isomorphism between $\Bir(X)$ and $\Bir(\PP^n)$ is always given by conjugation with a birational map $X\dashrightarrow \PP^n$ up to field automorphisms of the base-field \cite{deserti2006automorphismes}. This is not the case anymore as soon as $n\geq 3$ \cite{Zi2023Rigid-birational-i}.

It is an interesting question, which other rationality properties, such as stable rationality, unirationality or rational connectedness, of a variety $X$ are encoded in the group structure of $\Bir(X)$.  Our techniques show that $\Bir(X)$ characterizes ruledness:

\begin{theorem}
	\label{thm:charmain2}
	Let $X, Y$ be irreducible varieties such that $\Bir(X)$ and $\Bir(\PP^1 \times Y)$ 
	are isomorphic. Then $X$ is birationally equivalent to $\PP^1 \times Z$ for some variety $Z$.
\end{theorem}

For instance, we get from the non-ruledness of any smooth cubic threefold $X$
(see~\cite{ClGr1972The-intermediate-J}) that $\Bir(X)$ is non-isomorphic to 
$\Bir(\PP^1 \times Y)$ for any irreducible variety $Y$.

The \emph{Cremona group of rank $n$}, $\Bir(\PP^n)$, is often compared to linear groups. Indeed, $\Bir(\PP^n)$ satisfies several group theoretical properties typically satisfied by linear groups, such as the Jordan property 
 \cite{Se2009A-Minkowski-style-, Se2010The-Cremona-group-, PrSh2016Jordan-property-fo}, or the Tits alternative if $n=2$ \cite{Ca2011On-the-groups-of-b, Ur2021Subgroups-of-ellip}.  The groups $\Bir(X)$ for any variety $X$ can be equipped with the 
	\emph{Zariski topology} (Section~\ref{Sec.Notation_Preliminaries}), which makes them comparable to algebraic groups.
	An important tool to study linear algebraic groups are \emph{Borel subgroups}, i.e., maximal connected solvable subgroups with respect to the Zariski topology. In Section~\ref{sec.Borel}, we consider Borel subgroups of $\Bir(\PP^n)$ - a subject that has first been introduced in \cite{Po2017Borel-subgroups-of}.
In \cite{FuHe2023Borel-subgroups-of} the Borel subgroups of $\Bir(\PP^2)$ are completely classified. The \emph{standard Borel subgroup} $B_n$ of $\Bir(\PP^n)$ is the maximal connected solvable subgroup of derived length $2n$ consisting of transformations of the form 
\[
	(x_1,\dots, x_n)\dashmapsto (c_1x_1+d_1, \dots, c_n(x_1,\dots, x_{n-1})x_n+d(x_1,\dots, x_{n-1})),
\]
with respect to affine coordinates, where 
$c_i, d_i \in \CC(x_1, \ldots, x_{i-1})$ {and $c_i \neq 0$} (see \cite{Po2017Borel-subgroups-of} for details). We will prove the following result, which gives a new characterization of rational varieties among all varieties of dimension $\leq n$:

\begin{theorem}\label{thm:mainborel}
	Let $X$ be an irreducible variety of dimension $n$ and let $G\subset \Bir(X)$ be a 
	connected solvable subgroup. Then $G$ has derived length at most $2n$. 
	Moreover, if $G$ has derived length $2n$, then $X$ is birationally equivalent to $\PP^n$ and $G$ is conjugate to a subgroup of $B_n$. 
\end{theorem}

Theorem~\ref{thm:mainborel} has the following corollary, which answers a question of Popov \cite{Po2017Borel-subgroups-of} (see \cite{FuHe2023Borel-subgroups-of} for a proof for $n=2$):

\begin{corollary}
	\label{Cor.popov}
	For every $n \geq 2$, the Cremona group $\Bir(\PP^n)$ admits Borel subgroups
	of distinct derived length. In particular, $\Bir(\PP^n)$ contains Borel subgroups 
	that are non-conjugate to the standard Borel subgroup.
\end{corollary}

Analogously, we define \emph{Borel subgroups} of the group of polynomial automorphisms
$\Aut(\AA^n)$. The \emph{standard Borel subgroup of $\Aut(\AA^n)$} is $B_n\cap\Aut(\AA^n)$, and has derived length $n+1$  (see \cite{FuPo2018On-the-maximality-}). 
{For $n \leq 2$, all Borel subgroups of $\Aut(\AA^2)$ are conjugate to the standard Borel subgroup
\cite{BeEsEs2016Dixmier-groups-and}. However, with Theorem~\ref{thm:mainborel} we show the 
following corollary, which solves 
\cite[Problem~3]{FuPo2018On-the-maximality-}:}

\begin{corollary}
	\label{Cor.FurterPoloni}
	For $n\geq 3$, the group of polynomial automorphisms $\Aut(\AA^n)$ admits Borel subgroups that are non-conjugate to the standard Borel subgroup.
\end{corollary}

\subsection*{Outline of the proofs and structure of the article}
Throughout the paper we use results about the Zariski topology on $\Bir(X)$ and algebraic
families of birational transformations of $X$
from \cite{ReUrSa2024The-Structure-of-A},
which we recall in Section~\ref{Sec.Notation_Preliminaries}. 
Let us point out some new ingredients that go into the proofs of Theorem~\ref{thm:charmain},
Theorem~\ref{thm:charmain2} and Theorem~\ref{thm:mainborel}. 

A first one is a generalization of the celebrated theorem of Rosenlicht (Theorem~\ref{Thm.Rosenlicht}), which will be developed in Section~\ref{sec.Rosenlicht}. Recall that Rosenlicht's theorem states that if an algebraic group $G$ acts regularly on a variety $X$, then there exists a $G$-stable open dense $U\subseteq X$ admitting a geometric quotient \cite{Ro1956Some-basic-theorem}. We show that a similar result holds if we replace $G$ by a 
subgroup of $\Bir(X)$ generated by an irreducible algebraic family of birational maps 
$A\subset \Bir(X)$ that contains the identity
(Theorem~\ref{Thm.Rosenlicht}). 
While this is an interesting result on its own, it allows us to do certain induction steps by passing to the geometric generic fibre of this geometric quotient for certain families of birational transformations.

Another new ingredient is a result about the descent of unipotent elements
(Section~\ref{sec.Descent}). 
An element in $\Bir(X)$ is \emph{unipotent} if it is contained in a unipotent algebraic subgroup of $\Bir(X)$. We show that if an element $f\in\Bir(X)$ preserves a fibration $\pi\colon X\to Y$ such that the induced birational map on the geometric generic fibre of $\pi$ is unipotent, then $f$ is unipotent itself (Proposition~\ref{Prop.Descent_Ga}).

With this, we prove Theorem~\ref{thm:mainborel}
and the Corollaries~\ref{Cor.popov}, \ref{Cor.FurterPoloni}
in Section~\ref{sec.Borel}. For the proof of Theorem~\ref{thm:charmain} and Theorem~\ref{thm:charmain2} in Section~\ref{sec.Char} we
establish the following result
(Theorem~\ref{Thm.Existence_unipot}):
Let $G, H \subset \Bir(X)$ be closed connected commutative non-trivial
subgroups. Assume that $G$ normalizes $H$ and that $G$ acts by conjugation on $H$ ``essentially transitively''. Then the centralizer of  $H$ in $\Bir(X)$
contains a unipotent element. We apply this result to the connected components of the images 
of the standard torus and the subgroup of translations in $\Bir(\PP^n)$ under a 
group isomorphism $\Bir(\PP^n) \to \Bir(X)$. 

We state and prove all the results in the slightly more general setting of an algebraically closed uncountable base-field $\kk$ of characteristic $0$.  The uncountability of the base field is crucial in our proof. Indeed, we frequently use that the neutral component $G^\circ$ 
of a closed subgroup 
$G \subseteq \Bir(X)$ has countable index in $G$ and conclude with the uncountability of $G$ that $G^\circ$ is non-trivial. In particular, we apply this fact to the above described images in $\Bir(X)$
of the standard torus and the subgroup of translations.

\subsection*{Acknowledgements} We would like to thank J\'er\'emy Blanc, Michel Brion, and Hanspeter Kraft for many interesting discussions on the subject.
The first author is supported by DFG, project number 509752046.

\section{Notion and preliminary results}
\label{Sec.Notation_Preliminaries}

Throughout the article, all varieties, morphisms, and rational maps are defined over an algebraically closed
field $\kk$, $X$ is an irreducible variety over $\kk$, and $\Bir(X)$ its group of birational transformations. Here, a variety is a (not necessarily irreducible) reduced, separated scheme of finite type over $\kk$. Starting from Section~\ref{sec.Descent}, 
$\kk$ will be of 
characteristic zero and starting from Section~\ref{sec.Borel}, 
$\kk$ will be of characteristic zero and uncountable.

\subsection{Preliminary results on $\Bir(X)$}
We frequently use notions and results from our article~\cite{ReUrSa2024The-Structure-of-A}.
Let us briefly recall the most important ones.  For a birational map $\varphi\colon X\dashrightarrow Y$ we denote by $\lociso(\varphi)$ the open subset of $X$, where $\varphi$ is a local isomorphism. An \emph{algebraic family of birational transformations of $X$ pa\-ra\-met\-rized by a variety $V$}
is a $V$-birational transformation 
\[
	\theta \colon V \times X \dashrightarrow V \times X 
\]
such that $\lociso(\theta)$
surjects onto $V$. Such an algebraic family induces a map
$\rho_{\theta} \colon V \to \Bir(X)$, which we call a \emph{morphism to $\Bir(X)$}
and the image of a morphism we call an \emph{algebraic subset} of $\Bir(X)$.
The \emph{Zariski topology} on
$\Bir(X)$ is the finest topology such that all morphisms are continuous.
If $S \subseteq \Bir(X)$ is an algebraic subset, then its closure is algebraic as well,
see~\cite[Corollary~4.6]{ReUrSa2024The-Structure-of-A}. Images of
constructible subsets under morphisms to $\Bir(X)$ are again constructible 
\cite[Theorem~1.2]{ReUrSa2024The-Structure-of-A}.
With this topology, $\Bir(X) \times \Bir(X)$ is a closed subgroup of $\Bir(X \times X)$ and multiplication as well as taking the inverse are continuous maps. Moreover, 
$\Bir(X)$ is the union of an increasing chain of closed algebraic 
subsets~\cite[Theorem~1.1]{ReUrSa2024The-Structure-of-A}
and hence, the same holds for any closed subset of $\Bir(X)$.

For $S \subseteq \Bir(X)$, a subset $Z \subseteq X$ is called \emph{$S$-stable} if $s(z)\in Z$ for all $s\in S$ and all $x\in Z \cap \lociso(s)$.
If $x \in X$, then the $S$-orbit of $x$ is defined by
\[
	S x \coloneqq \set{s(x)}{\textrm{$s \in S$ such that $x \in \lociso(s)$}} \, .
\]
Note that $Sx$ can be the empty set.
If $S$ is a subgroup, then any two $S$-orbits are either disjoint or they coincide and the 
$S$-stable subsets are the unions of $S$-orbits. Moreover, in this case the singular locus of $X$ is 
$S$-stable.

For a subset $S \subseteq \Bir(X)$ we define its \emph{dimension} $\dim(S)$ as the supremum of
$\dim(V)$ over all injective morphisms $V \to \Bir(X)$ with image in $S$. 
For a closed algebraic subset of $\Bir(X)$ its dimension coincides with the Krull-dimension
induced from the topology on $\Bir(X)$, see~\cite[Theorem~1.4]{ReUrSa2024The-Structure-of-A}. 
For any morphism $\rho \colon V \to \Bir(X)$ from an irreducible variety $V$ 
we have the following \emph{fibre-dimension formula} (see~\cite[Theorem~1.5]{ReUrSa2024The-Structure-of-A}): There exists an open dense subset $U \subseteq V$
such that
\[
	\dim_u (U \cap \rho^{-1}(\rho(u))) = 
	\dim V - \dim \overline{\rho(V)} 
	\quad \quad \textrm{for all $u \in U$} \, ,
\]
where $\dim_u$ denotes the local dimension at $u$.

A subgroup of $\Bir(X)$ is called \emph{algebraic}, if it is the image of a morphism
from an algebraic group and the morphism is a group homomorphism. Note that algebraic
subgroups are always closed in $\Bir(X)$ \cite[Corollary~5.12]{ReUrSa2024The-Structure-of-A}.
Moreover, any closed finite-dimensional connected subgroup $G$ of $\Bir(X)$ has a unique structure
of an algebraic group such that
for any algebraic group $H$, the algebraic group homomorphisms $H \to G$ correspond 
bijectively to the morphisms $H \to \Bir(X)$ that are group homomorphisms and have
their image in $G$ \cite[Theorem~1.3]{ReUrSa2024The-Structure-of-A}.

A closed subgroup $G \subseteq \Bir(X)$ can always be filtered by 
closed algebraic subsets $G_1 \subseteq G_2 \subseteq \cdots$ of $\Bir(X)$ 
such that the irreducible components of $G_i$ are pairwise disjoint and homeomorphic;
if $G$ is also connected, then the $G_i$ can be chosen to be 
irreducible~\cite[Corollary~4.12]{ReUrSa2024The-Structure-of-A}.
Moreover, for a closed connected subgroup $G$ of $\Bir(X)$, its connected component $G^\circ$
of the neutral element has countable index in 
$G$ and is open (thus closed) in $G$~\cite[Corollary~4.11]{ReUrSa2024The-Structure-of-A}.

Following~\cite[Definition~1, p. 514]{De1970Sous-groupes-algeb},
for an algebraic group $G$ a rational map 
$\alpha \colon G \times X \dashrightarrow X$ is called a
\emph{rational $G$-action} if $\theta \colon G \times X \dashrightarrow G \times X$, 
$(g, x) \dashmapsto (g, \alpha(g, x))$ is dominant and the following diagram commutes:
\[
	\xymatrix@=15pt{
		G \times G \times X \ar@{-->}[d]_-{\id_G \times \alpha} \ar[rrrr]^-{(g_1, g_2, x) \mapsto (g_1 g_2, x)} 
		&&&& G \times X \ar@{-->}[d]^-{\alpha} \\
		G \times X \ar@{-->}[rrrr]^-{\alpha} &&&& X \, .
	}
\]
The rational map $\theta$ is then an algebraic family parametrized by $G$
and the induced morphism $\rho_{\theta} \colon G \to \Bir(X)$ is a group homomorphism, 
see \cite[Corollary~3.3]{ReUrSa2024The-Structure-of-A}.
By abuse of notation, we sometimes write $\rho_\alpha$ instead of $\rho_{\theta}$.
We denote by $\GG_a$ and $\GG_m$ the algebraic groups of the underlying additive and multiplicative
group of the ground field $\kk$, respectively. An element in $\Bir(X)$ that is contained
in the image of a $\GG_a$-action is called \emph{unipotent} and an algebraic subgroup
of $\Bir(X)$ is called \emph{unipotent} if every element is unipotent.

In this paragraph, we assume that the characteristic of $\kk$ is zero.
We recall two continuity statements that can be found 
in~\cite[Theorem~1.6]{ReUrSa2024The-Structure-of-A}.
Let $\pi \colon X \to Y$ be a dominant morphism of irreducible varieties.
We denote by $\Bir(X,\pi)$ the subgroup of elements in $\Bir(X)$
that induce a birational map on $Y$ and by $\Bir(X/Y)$ the subgroup
of elements that induce the identity on $Y$. Then, the natural homomorphism 
$\Bir(X,\pi) \to \Bir(Y)$ is continuous. 
Let $K$ be an algebraic closure of the function field $\kk(Y)$. We denote by $X_K$
the geometric generic fibre of $\pi \colon X \to Y$. 
If $X_K$ is an irreducible $K$-variety (which happens e.g. if $\kk(Y)$ is algebraically closed in $\kk(X)$ and $\car(\kk) = 0$), 
then the natural pull-back homomorphism
\[
	\Bir(X/Y) \to \Bir_K(X_K) \, , \quad \varphi \mapsto \varphi_K
\]
is continuous, where $\Bir_K(X_K)$ denotes the group of $K$-birational transformations
of $X_K$. Moreover, for any subset $S \subseteq \Bir(X/Y)$ we have
\begin{equation}
	\label{Eq.invariants}
	K(X_K)^{\overline{\varepsilon(S)}} = \kk(X)^S \otimes_{\kk(Y)} K \, ,
\end{equation}
where we used the fact that the function field $K(X_K)$ of $X_K$ over $K$ is given by
$\kk(X) \otimes_{\kk(Y)} K$ \cite[Ch. 3, Corollary~2.14(c)]{Li2002Algebraic-geometry} 
and the fact that the  invariant field with respect to a subset of birational
transformations does not change if we take the closure of that subset 
\cite[Corollary~7.4]{ReUrSa2024The-Structure-of-A}.

\subsection{Preliminary results on algebraic groups}

Let $H$ be a connected algebraic group. Then 
$H$ contains a unique normal connected affine algebraic subgroup
$H_{\aff}$ that contains every connected affine algebraic subgroup of $H$, 
see \cite[Theorem 16, p. 439]{Ro1956Some-basic-theorem}. Moreover, the smallest normal closed subgroup  $H_{\ant}$ of $H$ such that $H / H_{\ant}$
is affine satisfies $H_{\aff} H_{\ant} = H$, see~\cite{Br2009Anti-affine-algebr}.
If $H$ is commutative, then the unipotent elements in $H$ form a closed subgroup that is 
contained in $H_{\aff}$, see~\cite[Theorem~15.5]{Hu1975Linear-algebraic-g}.

We denote by $\Aut_{\textrm{alg.grp}}(H)$
the group of algebraic group automorphisms of $H$.

\begin{lemma}
	\label{Lem.Structure_Aut_alg.grp}
	Assume that $H$ is commutative. Then $\Aut_{\textrm{alg.grp}}(H)$
	contains a countable index subgroup that is isomorphic to a subgroup of $\GL_n$, 
	where $n$ is the dimension
	of the subgroup of unipotent elements of $H$.
	In particular, if $\Aut_{\textrm{alg.grp}}(H)$ is uncountable, 
	then $H$ contains non-trivial unipotent elements.
\end{lemma}

\begin{proof}
	We get an injection
	$\Aut_{\textrm{alg.grp}}(H) \hookrightarrow \Aut_{\textrm{alg.grp}}(H_{\aff}) \times 
	\Aut_{\textrm{alg.grp}}(H_{\ant})$. As the group $\Aut_{\textrm{alg.grp}}(H_{\ant})$ 
	is countable \cite[Lemma~1.5]{Br2009Anti-affine-algebr},
	we only have to consider the group $\Aut_{\textrm{alg.grp}}(H_{\aff})$.
	Let $H_{\aff, s}$ and $H_{\aff, u}$ be the subgroups of semi-simple and unipotent elements of $H_{\aff}$,
	respectively \cite[Theorem~15.5]{Hu1975Linear-algebraic-g}. 
	Then  $\Aut_{\textrm{alg.grp}}(H_{\aff})$ is the product of
	the algebraic group automorphisms of $H_{\aff, s}$ and $H_{\aff, u}$. Note that
	$\Aut_{\textrm{alg.grp}}(H_{\aff, s})$ is countable and $\Aut_{\textrm{alg.grp}}(H_{\aff, u})$
	is isomorphic to $\GL_n$, where $n = \dim H_{\aff, u}$. This implies the statement.
\end{proof}

\section{\texorpdfstring{Rosenlicht quotients for subgroups of $\Bir(X)$}{Rosenlicht quotients for subgroups of Bir(X)}}
\label{sec.Rosenlicht}

In this section, $\kk$ denotes an algebraically closed field, we do not impose any restriction on the
characteristic or the cardinality.
The goal of this section is to prove a generalization of the classical Rosenlicht theorem
for certain subgroups of $\Bir(X)$
and to derive some consequences of it, in particular to commutative subgroup with trivial
invariant field, see Corollary~\ref{Cor.Centralizer}. 
Our proof is inspired by~\cite[\S14.5]{FeRi2017Actions-and-invari}.

\begin{definition}
	Let $G \subseteq \Bir(X)$ be a subgroup. A morphism $\pi \colon X \to Y$ is called
	a \emph{geometric quotient} for $G$, if $\pi$ is surjective, open, its fibres
	are $G$-orbits, and for every open subset $U \subseteq Y$ the morphism $\pi$
	induces an isomorphism $\OO_Y(U) \xrightarrow{\sim} \OO_X(\pi^{-1}(U))^G$
	to the $G$-invariant functions on $\pi^{-1}(U)$.
	A morphism $\pi_0 \colon X_0 \to Y$ is called \emph{Rosenlicht quotient} for $G$, if
	$X_0$ is an open, dense $G$-stable subset of $X$ and $\pi_0$ is a geometric quotient for $G$.
\end{definition}

The last property of a geometric quotient for $G$ is equivalent to $\kk(Y) = \kk(X)^G$
provided that $Y$ is normal:

\begin{lemma}
	\label{Lem.geometric_quotient}
	Let $G \subseteq \Bir(X)$ be a subgroup and let $\pi \colon X \to Y$ be a
	surjective, open morphism to a normal variety $Y$ 
	such that its fibres are $G$-orbits.  Then $\pi$
	is a geometric quotient for $G$ if and only if $\kk(Y) = \kk(X)^G$.
\end{lemma}

\begin{proof}
	Let $U \subseteq Y$ be an open subset and let $f \colon \pi^{-1}(U) \to \AA^1$ be a 
	$G$-invariant function. Since $\pi^{-1}(U) \to U$ is open, surjective and its 
	fibres are $G$-orbits, there exists a continuous map $\bar{f} \colon U \to \AA^1$
	with $f = \bar{f} \circ \pi |_{\pi^{-1}(U)}$. As $\kk(X)^G = \kk(Y)$,
	$\bar{f}$ is rational. Using that $Y$ is normal we conclude that $\bar{f}$ is a morphism
	(see e.g.~\cite[Lemma~7.6]{ReUrSa2024The-Structure-of-A}).
	For the reverse implication it is enough to note that the domain of a 
	$G$-invariant rational map is $G$-stable.
\end{proof}

\begin{theorem}
	\label{Thm.Rosenlicht}
	Let $G \subseteq \Bir(X)$ be a subgroup that is generated by
	countably many irreducible algebraic subsets of $\Bir(X)$, where each of them contains 
	the identity.
	Then $G$ admits a Rosenlicht quotient with irreducible fibres and
	$\kk(X)^G$ is algebraically closed in $\kk(X)$.
\end{theorem}

Note that a closed connected subgroup of $\Bir(X)$ satisfies the assumptions of 
Theorem~\ref{Thm.Rosenlicht}, as it can be filtered 
by irreducible closed algebraic subsets of $\Bir(X)$.
In particular, this holds if $G$ is a connected so-called \emph{ind-group} that 
acts regularly on $X$, 
a fact that is also proved by Kraft, \cite{Kr2024On-Actions-and-Quo}.

If $\car(\kk) = 0$, the last statement of Theorem~\ref{Thm.Rosenlicht}
implies that the geometric generic fibre of $\pi$ is integral, see 
e.g.~\cite[Proposition~5.51]{GoWe0Algebraic-geometryI}.

\medskip

For the proof of Theorem~\ref{Thm.Rosenlicht} we take
an increasing sequence of closed irreducible subvarieties $V_1 \subseteq V_2 \subseteq \cdots$
and a map $\rho \colon \bigcup_{d \geq 1} V_d \to \Bir(X)$ with image equal to $G$
such that $\rho_d \coloneqq \rho |_{V_d} \colon V_d \to \Bir(X)$ is a morphism for each $d \geq 1$,
see Corollary~\cite[Corollary~4.13]{ReUrSa2024The-Structure-of-A}.
Let $\theta_d \colon V_d \times X \dashrightarrow V_d \times X$ 
be the algebraic family corresponding to $\rho_d$. 

\begin{remark}\label{Rem.NiceOrbit}
	For $x \in X$ we consider the following subset of $X$
	\[
		V_d^\ast x \coloneqq (\pr_X \circ \theta_d)(\lociso(\theta_d) \cap (V_d \times \{x\}) )  \, ,
	\]
	where $\pr_X$ denotes the projection to $X$.
	Then $V_d^\ast x$ is an irreducible constructible subset of $X$ that is contained in 
	$\rho_d(V_d)x$. The proof of Theorem~\ref{Thm.Rosenlicht} will show that 
	there exists $n \geq 1$ such that $V_n^\ast x$ is dense in the orbit $Gx$ for general $x \in X$.
\end{remark}

For $f \in \kk(X)$ and $g \in G$  we denote by $f^g$ the rational function
$x \dashrightarrow f(g(x))$. We start the proof of Theorem~\ref{Thm.Rosenlicht} 
with two lemmas:

\begin{lemma}
	\label{Lem.Lin_indep_overkXG}
	If $f_1, \ldots, f_s \in \kk(X)$ are linearly independent over $\kk(X)^G$, then 
	we have for all  $q_1, \ldots, q_s \in \kk(X)$
	\[
		\sum_{i=1}^s q_i \cdot f_i^g = 0 \quad \textrm{for all $g \in G$}
		\quad \implies \quad q_1 = \ldots = q_s = 0 \, ,
	\]
\end{lemma}

\begin{proof}
	If $s = 1$, then the statement follows from $f_1 \neq 0$, and thus we assume $s > 1$.
	Assume towards a contradiction that $q_1 \neq 0$, and hence,
	\begin{equation}
		\label{Eq1}
		\sum_{i=1}^s \frac{q_i}{q_1} \cdot f_i^g = 0 \quad \textrm{for all $g \in G$} \, .
	\end{equation}
	This implies $\sum_{i=1}^s \frac{q_i^a}{q_1^a} \cdot f_i^{ba} = 0$ for all $a, b \in G$ and thus
	\begin{equation}
		\label{Eq2}
		\sum_{i=1}^s \frac{q_i^a}{q_1^a} \cdot f_i^g = 0 \quad \textrm{for all $a, g \in G$}  \, .
	\end{equation}
	Now equations~\eqref{Eq1} and~\eqref{Eq2} give for all $a, g \in G$
	\[
		0 = \sum_{i=1}^s \frac{q_i^a}{q_1^a} \cdot f_i^g  - \sum_{i=1}^s \frac{q_i}{q_1} \cdot 
		f_i^g = \sum_{i=2}^s \left( \frac{q_i^a}{q_1^a} - \frac{q_i}{q_1}  \right) 
		\cdot f_i^g \, .
	\]
	By induction hypothesis we get $\frac{q_i}{q_1} \in \kk(X)^G$ for all $i=2, \ldots, s$.
	Now, equation~\eqref{Eq1} applied to $g = \id_X$ gives 
	$\frac{q_i}{q_1} = 0$ for all $i \geq 1$, since 
	$f_1, \ldots, f_s$ are linearly independent over $\kk(X)^G$, contradiction.
\end{proof}

\begin{lemma}
	\label{Lem.KXG_alg_closed}
	The field $\kk(X)^G$ is algebraically closed in $\kk(X)$.
\end{lemma}

\begin{proof}
	Let $f \in \kk(X)$ be algebraic over $\kk(X)^G$. Then there exists
	a non-zero univariate polynomial $p$ with coefficients in  $\kk(X)^G$ such that 
	$p(f) = 0$ in $\kk(X)$
	and hence $p(f^g) = 0$ for all $g \in G$. In particular, 
	$S = \set{f^g \in \kk(X)}{g \in G}$ is finite. 
	We have to prove that $S$ consists of a single element.
	Towards a contradiction, we assume that there  are $g_0, g_1 \in G$ with
	$f^{g_0} \neq f^{g_1}$. Take $d \geq 1$ and $v_0, v_1 \in V_d$ with $g_i = \rho_d(v_i)$.
	There exists $x \in X$ with $(v_i, x) \in \lociso(\theta_d)$,
	$g_i(x) \in \dom(f)$ for $i=1,2$  and  $f^{g_0}(x) \neq f^{g_1}(x)$. Note that
	\[
		T \coloneqq  f(V_d^\ast x \cap \dom(f)) 
		\subseteq \AA^1
	\]
	is a finite set, since $S$ is finite. However, as $T$ is irreducible, 
	we get a contradiction.
\end{proof}

\begin{proof}[Proof of Theorem~\ref{Thm.Rosenlicht}]
	By \cite[Theorem~24.9]{Is2009Algebra:-a-graduat} there exist finitely many generators 
	$p_1, \ldots, p_r$ of the field extension $\kk \subseteq \kk(X)^G$. Consider the rational map
	\[
		\pi = (p_1, \ldots, p_r) \colon X \dashrightarrow Y \subseteq \AA^r \, ,
	\]
	where $Y$ is the closure of the image of $\pi$ in $\AA^r$.
	By construction, $\kk(Y) = \kk(X)^{G}$.
	After shrinking $X$, we may assume that
	$\pi$ is everywhere defined 
	(note that in this process, we only remove a union of $G$-orbits). 
	Moreover, by further shrinking $Y$ (and hence $X$), we may assume that
	$\pi$ is flat and surjective, and $Y$ is affine and smooth
	(again, we only remove a union of $G$-orbits in $X$).
	Hence, using Lemma~\ref{Lem.geometric_quotient} and Lemma~\ref{Lem.KXG_alg_closed} it is enough to show that the 
	fibres of $\pi$ are irreducible $G$-orbits.

	\medskip

	For proving this, let $U_d = \lociso(\theta_d)$ and consider the morphism
	\[
		\Phi_d \colon U_d \to X \times_Y X \, , \quad
		(v, x) \mapsto ((\pr_X \circ \theta_d)(v, x), x) \, .
	\]
	Let $Z_d \coloneqq \Phi_d(U_d)$. Since 
	$\rho_d = \rho_{d+1} |_{V_d}$, we have 
	$\theta_{d}(v, x) = \theta_{d+1}(v, x)$ for all $(v, x) \in U_{d+1} \cap V_d \times X$.
	This implies that $\overline{Z_d} \subseteq \overline{Z_{d+1}}$ for all $d \geq 1$.
	Hence, there exists 
	$n \in \NN_0$ with $\overline{Z_n} = \overline{Z_{n+i}}$ for all $i \geq 0$. 
	
	\begin{claim}
		\label{Claim1}
		The morphism $\Phi_n \colon U_n \to X \times_Y X$ is dominant.
	\end{claim}
	
	\begin{proof}
		Take open dense affine subsets $W_1$, $W_2$ in $X$. To prove the claim it is enough 
		to show that 
		\[
			\Phi_n^{-1}(W_1 \times_Y W_2) \xrightarrow{\Phi_n} W_1 \times_Y W_2
		\]
		is dominant. 
		Let $h_i \in \OO(W_1)$, $f_i \in \OO(W_2)$, $i=1, \ldots, s$ such that 
		\begin{equation}
			\label{Eq.Phi}
			\Phi_n^{-1}(W_1 \times_Y W_2) \xrightarrow{\Phi_n} W_1 \times_Y W_2 \xrightarrow{\sum_{i=1}^s h_i \otimes f_i } \AA^1
		\end{equation}
		is the zero map. It is enough to show that $\sum_{i=1}^s h_i \otimes f_i$ is zero in $\OO(W_1) \otimes_{\OO(Y)} \OO(W_2)$.
		As $\OO(W_1), \OO(W_2)$ are flat $\OO(Y)$-modules, we get a natural injection
		\[
		\OO(W_1) \otimes_{\OO(Y)} \OO(W_2) \to \kk(X) \otimes_{\kk(Y)} \kk(X) \, .
		\]
		After reordering $f_1, \ldots, f_s$, we may assume that $f_1, \ldots, f_t$ are linearly independent over $\kk(Y)$ and every $f_i$ can be written as
		$f_i = \sum_{j=1}^t \frac{b_{ij}}{c} f_j$,
		for $b_{ij}, c \in \OO(Y)$ and $c \neq 0$. Note that
		\[
			\sum_{i=1}^s h_i \otimes f_i = \sum_{i=1}^s h_i \otimes \sum_{j=1}^t \frac{b_{ij}}{c} f_j
			= \frac{1}{c} \sum_{j=1}^t \left(\sum_{i=1}^s b_{ij} h_i \right) \otimes f_j \, .
		\]
		Hence, we may and will assume that $f_1, \ldots, f_s$ are linearly independent over $\kk(Y)$.

		\smallskip

		Using Lemma~\ref{Lem.Lin_indep_overkXG}, Claim~\ref{Claim1} follows once we have proven
		that $\sum_{i=1}^s h_i \cdot f_i^g = 0$ in $\kk(X)$ for all $g \in G$.
		Let $m \geq n$. As $\overline{Z_m} = \overline{Z_n}$,
		\[
				\Phi_m^{-1}(W_1 \times_Y W_2) \xrightarrow{\Phi_m} W_1 \times_Y W_2 \xrightarrow{\sum_{i=1}^s h_i \otimes f_i} \AA^1
		\]
		is the zero map. Hence $\sum_{i=1}^s h_i(x) \otimes f_i((\pr_X \circ \theta_m)(v, x)) = 0$
		for all $(v, x) \in \Phi_m^{-1}(W_1 \times_Y W_2)$. 
		Since $\Phi_m^{-1}(W_1 \times_Y W_2)$ surjects onto $V_m$
		(for $v \in V_m$ take $w_2 \in W_2$ such that $(v, w_2) \in U_m$ and 
		$(\pr_X \circ \theta_m)(v, w_2) \in W_1$; then $\Phi_m(v, w_2) \in W_1 \times_Y W_2$), 
		we obtain $\sum_{i=1}^s h_i \cdot f_i^g = 0$ for all $g \in \rho_m(V_m)$. 
		As $G = \bigcup_{m \geq 1} \rho_m(V_m)$, this implies
		$\sum_{i=1}^s h_i \cdot f_i^g = 0$ for all $g \in G$
	\end{proof}

	For $i = 1, 2$, let $\pr_i \colon X \times_Y X \to X$ be the projection to the $i$-th factor.
	Since $Z_n$ is constructible and dense in $X \times_Y X$ (see Claim~\ref{Claim1}),
	there exists an open dense subset $X_1 \subseteq X$ such that $\pr_2^{-1}(x) \cap Z_n$
	is dense in $\pr_2^{-1}(x)$ for all $x \in X_1$ (see e.g.~\cite[Lemma~3.4]{ReUrSa2024The-Structure-of-A}).
	Hence, the subset $V_n^\ast x = \pr_1(\pr_2^{-1}(x) \cap Z_n)$ is 
	dense in $\pi^{-1}(\pi(x)) = \pr_1(\pr_2^{-1}(x))$ for all $x \in X_1$.
	This will imply the statement of Remark~\ref{Rem.NiceOrbit} 
	once we show that $\pi^{-1}(\pi(x))$ is a $G$-orbit 
	for general $x \in X$. Let
	\[
		X_0 \coloneqq \bigcup_{g \in G} g(\lociso(g) \cap X_1) \subseteq X \, , \quad
		\pi_0 \coloneqq \pi |_{X_0} \colon X_0 \to \pi(X_0) \, .
	\]
	Note that $X_0$ is an open dense $G$-stable subset of $X$ that contains $X_1$.
	Since $\pi$ is constant on $G$-orbits, we obtain
	$\pi(X_0) = \pi(X_1)$.
	
	Let $x_1 \in X_1$. Then $V_n^\ast x_1 \subseteq G x_1 \subseteq \pi_0^{-1}(\pi_0(x_1))$ and
	$V_n^\ast x_1$ is constructible and 
	dense in the fibre $\pi_0^{-1}(\pi_0(x_1))$. Since $V_n^\ast x_1$ is irreducible,
	$\pi_0^{-1}(\pi_0(x_1))$ is irreducible as well.
	Assume there exists a point $x_1' \in \pi_0^{-1}(\pi_0(x_1)) \setminus G x_1$. As $X_0$
	is a union of $G$-orbits of points in $X_1$, we may assume that $x_1' \in X_1$. Again $V_n^\ast x_1'$ is constructible and dense in
	$\pi_0^{-1}(\pi_0(x_1')) =  \pi_0^{-1}(\pi_0(x_1))$. Hence,
	$G x_1$ and $G x_1'$ intersect, contradiction.
\end{proof}

Using Remark~\ref{Rem.NiceOrbit} we get the following consequence:

\begin{corollary}
	\label{Cor.NiceOrbit}
	Assume that $G \subseteq \Bir(X)$ satisfies the hypothesis of Theorem~\ref{Thm.Rosenlicht}.
	Then $G$ contains an irreducible 
	algebraic subset $Z \subseteq \Bir(X)$
	such that $Z x$ is dense in $G x$ for general $x \in X$ and $Z x$ contains 
	a subset that is open in $G x$. 
	
	If moreover, $G$ is an algebraic subgroup, then for general $x \in X$
	there is an open, dense subset $V_x \subseteq G$ and 
	a dominant morphism $V_x \to G x$ given by $v \mapsto v(x)$.
\end{corollary}

The following is a direct application of Theorem~\ref{Thm.Rosenlicht}:

\begin{corollary}
	\label{Cor.Trivial_invariants}
	Assume that $G \subseteq \Bir(X)$ satisfies the hypothesis of Theorem~\ref{Thm.Rosenlicht}.
	Then $\kk(X)^G = \kk$ if and only if $X$ admits a dense open $G$-orbit.
\end{corollary}

As a further application we prove a 
generalization of \cite[Lemma~2.10]{KrReSa2021Is-the-affine-spac}:

\begin{corollary}
	\label{Cor.Centralizer}
	Let $G \subseteq \Bir(X)$ be a commutative 
	subgroup with $\kk(X)^G = \kk$ 
	that satisfies the assumptions of Theorem~\ref{Thm.Rosenlicht}.
	Then the centralizer of $G$ in $\Bir(X)$ is equal to $G$
	and $G$ is a connected algebraic subgroup of $\Bir(X)$ of dimension $\dim X$.
\end{corollary}

\begin{lemma}
	\label{Lem.Dense_open_orbit}
	Let $Z \subseteq \Bir(X)$ be an irreducible algebraic subset such that there exist
	$x_0 \in X$ and an open dense subset $U \subseteq X$ with $U \subseteq Z x_0$.
	Moreover, assume that the elements of $Z$ commute pairwise. 
	Then:
	\begin{enumerate}[left=0pt]
		\item \label{Lem.Dense_open_orbit1} $\Cent_{\Bir(X)}(Z)$ is contained in $Z \circ Z^{-1}$.
		\item \label{Lem.Dense_open_orbit2} 
			If $x_0 \in \lociso(\varphi)$ for some $\varphi \in \Cent_{\Bir(X)}(Z)$
			and $\varphi(x_0) = x_0$, then $\varphi = \id_X$.
	\end{enumerate}
	
\end{lemma}

\begin{proof}
	Let $\varphi_0 \in \Cent_{\Bir(X)}(Z)$. As 
	$\varphi_0(\lociso(\varphi_0) \cap U) \cap U$
	is non-empty, there exist $f, h \in Z$  such that 
	$x_0 \in \lociso(f) \cap \lociso(h)$, $h(x_0) \in \lociso(\varphi_0)$ 
	and $f(x_0) = \varphi_0(h(x_0))$. 
	Hence, $f^{-1} \circ \varphi_0 \circ h$ is a local isomorphism at $x_0$
	and maps $x_0$ onto itself. In order to show~\eqref{Lem.Dense_open_orbit1}
	it is thus enough to show~\eqref{Lem.Dense_open_orbit2} 
	(as $f^{-1} \circ \varphi_0 \circ h \in \Cent_{\Bir(X)}(Z)$).

	Let $\varphi \in \Cent_{\Bir(X)}(Z)$ such that $\varphi$ is a local isomorphism at $x_0$
	and $\varphi(x_0) = x_0$. Take
	$u \in U \cap \lociso(\varphi)$. There exists 
	$g \in Z$ with $x_0 \in \lociso(g)$ and $u = g(x_0)$. Hence,
	$\varphi(u) = \varphi(g(x_0)) = g(\varphi(x_0)) = g(x_0) = u$.
	This shows that $\varphi$ is the identity on $X$
	and hence~\eqref{Lem.Dense_open_orbit2}.
\end{proof}

\begin{proof}[Proof of Corollary~\ref{Cor.Centralizer}]
	Since $\kk = \kk(X)^{G}$, there exists an open dense $G$-orbit in $X$,
	see Corollary~\ref{Cor.NiceOrbit}. Using Corollary~\ref{Cor.NiceOrbit} there exist
	$x_0 \in X$ and an irreducible algebraic subset $Z \subseteq \Bir(X)$ that is contained in $G$
	such that $Z x_0$ contains an open dense subset of $X$. Hence, we may apply 
	Lemma~\ref{Lem.Dense_open_orbit}\eqref{Lem.Dense_open_orbit1}
	to $Z$ and $x_0$ in order to conclude 
	\[
		G \subseteq \Cent_{\Bir(X)}(G) \subseteq \Cent_{\Bir(X)}(Z) 
		\subseteq Z \circ Z^{-1} \subseteq G \, .
	\]
	This shows that $G = \Cent_{\Bir(X)}(G)$ is an irreducible 
	algebraic subset of $\Bir(X)$ and hence $G$ is a closed connected 
	algebraic subgroup of $\Bir(X)$. Again by Corollary~\ref{Cor.NiceOrbit}, 
	there exist $x_1 \in G x_0$ in the open $G$-orbit
	and an open dense subset $V \subseteq G$ such that $V \to G x_0$, $g \mapsto g(x_1)$
	is a dominant morphism.	
	By Lemma~\ref{Lem.Dense_open_orbit}\eqref{Lem.Dense_open_orbit2} this morphism is injective
	and thus $\dim X = \dim G$.
\end{proof}

Let $\pi \colon X \to Y$ be a dominant morphism of irreducible varieties
such that the geometric generic fibre $X_K$ is an irreducible $K$-variety, 
where $K$ denotes an algebraic closure of $\kk(Y)$. In the following version of 
Corollary~\ref{Cor.Centralizer}
we denote by $\varepsilon \colon \Bir(X/Y) \to \Bir_K(X_K)$ the associated pull-back homomorphism.

\begin{corollary}
	\label{Cor.Centralizer_pull-back}
	Let $\pi \colon X \to Y$ be a Rosenlicht quotient
	for a closed connected commutative subgroup $H \subseteq \Bir(X)$ 
	such that the geometric generic fibre $X_K$ 
	is an irreducible $K$-variety (which always holds if $\car(\kk) = 0$). 
	Then $\overline{\varepsilon(H)} \subseteq \Bir_K(X_K)$ is a connected, self-centra\-lizing
	algebraic subgroup of dimension $\dim X - \dim Y$.
\end{corollary}

\begin{proof}
	Note that $\overline{\varepsilon(H)}$ is connected by the continuity of $\varepsilon$
	and that the invariant field of $K(X_K)$ under 
	$\overline{\varepsilon(H)}$ is equal to $K$ by~\eqref{Eq.invariants}. 
	Hence, Corollary~\ref{Cor.Centralizer} implies the statement.
\end{proof}

\section{Descent of unipotent elements}
\label{sec.Descent}
In this section $\kk$ denotes an algebraically closed field of characteristic zero.
Moreover, denote by $\pi \colon X \to Y$ a dominant morphism of irreducible varieties,
let $K$ be an algebraic closure of $\kk(Y)$ and  denote by $X_K$ the 
geometric generic fibre of $\pi$.
We assume that $X_K$ is an irreducible $K$-variety.
The goal of this section is a descent result for unipotent algebraic subgroups in birational transformation groups.

\medskip

Let $\Gamma = \Aut_{\kk(Y)}(K)$ be the Galois group of the Galois extension 
$\kk(Y) \subseteq K$. For any $K$-variety $V$ and $\tau \in \Gamma$ 
we consider the pull-back of $V \to \Spec(K)$
via $\Spec(\tau)$
\[
	\xymatrix@R=15pt{
		V^{\tau} \ar[d] \ar[rr]^-{\tau_{V}} && V \ar[d] \\
		\Spec(K) \ar[rr]^-{\Spec(\tau)} && \Spec(K)
	}
\]
in the category of schemes 
($\tau_{V}$ and $\Spec(\tau)$ are only isomorphisms of schemes).
If $V$ is the pull-back of a $\kk(Y)$-variety to $K$, then we may and will identify 
$V^\tau$ and $V$. For example, $\tau_{X_K}$ is an automorphism of schemes of $X_K$,
as $X_K$ is the pull-back of the generic fibre $X_{\kk(Y)} \to \Spec(\kk(Y))$ to $K$.
We get a natural right-action of $\Gamma$ on $\Bir_K(X_K)$ that is given by
\[
	\varphi^\tau \coloneqq \tau_{X_K} \circ \varphi \circ \tau_{X_K}^{-1} \, ,
\]
where $\tau \in \Gamma$ and $\varphi \in \Bir_K(X_K)$. Then, the fixed points in 
$\Bir_K(X_K)$ under this $\Gamma$-action is equal to the image of the pull-back
homomorphism $\Bir(X/Y) \to \Bir_K(X_K)$, since the function field $K(X_K)$ is
the tensor product $\kk(X) \otimes_{\kk(Y)} K$. 
Moreover, if $\tau \in \Gamma$ and 
$\rho \colon V \to \Bir_K(X_K)$ is a morphism corresponding to an algebraic family
$\theta$ of $K$-birational transformations of $X_K$ parametrized by $V$, we denote by 
$\rho^{\tau} \colon V^\tau \to \Bir_K(X_K)$ the morphism associated to the algebraic 
family $(\tau_V \times \tau_{X_K})^{-1} \circ \theta \circ (\tau_V \times \tau_{X_K})$.
Thus, we get the following commutative diagram
\begin{equation}
	\begin{gathered}
	\label{Eq.Alg_families_and_Galois_action}
	\xymatrix@R=15pt{
		V^\tau \ar[d]_-{\tau_V} \ar[r]^-{\rho^{\tau}} & \Bir_K(X_K) 
		\ar[d]^{\varphi \mapsto \varphi^{\tau}} \\
		V \ar[r]^-{\rho} & \Bir_K(X_K) \, .
	}
	\end{gathered}
\end{equation}
In particular, for every $\tau \in \Gamma$, 
the group automorphism $\varphi \mapsto \varphi^\tau$ of $\Bir_K(X_K)$ is a homeomorphism, 
i.e.~$\Gamma$ acts through continuous group automorphisms on $\Bir_K(X_K)$.

If $G$ is an algebraic subgroup of $\Bir_K(X_K)$ that is preserved by $\Gamma$ and if 
$\rho \colon G \to \Bir_K(X_K)$ is the corresponding injective morphism, then we get
for all $\tau \in \Gamma$ a unique isomorphism (over $K$) of algebraic groups 
$\iota_\tau \colon G^\tau \simeq G$ such that $\rho^\tau = \rho \circ \iota_\tau$
by using~\eqref{Eq.Alg_families_and_Galois_action} 
and the universal property of algebraic subgroups of birational transformation
groups. In particular, we may identify $\rho^\tau \colon G^\tau \to \Bir_K(X_K)$ and
$\rho \colon G \to \Bir_K(X_K)$, and we get for all $\tau \in \Gamma$
the following commutative diagram
\begin{equation}
	\label{Eq.Alg_subgroup_of_Bir_preserved_by_Gamma}
	\begin{gathered}
	\xymatrix@R=15pt{
		G \ar[d]_-{\tau_G} \ar[r]^-{\rho} & \Bir_K(X_K) 
		\ar[d]^{\varphi \mapsto \varphi^{\tau}} \\
		G \ar[r]^-{\rho} & \Bir_K(X_K) \, .
	}
	\end{gathered}
\end{equation}
Hence, $\Gamma$ acts naturally on the right on $G$ through $\kk(Y)$-isomorphisms that are
group automorphisms, and $\rho \colon G \to \Bir_K(X_K)$
is then $\Gamma$-equivariant. If $G$ is affine, then the $\Gamma$-action on $G$ induces 
a left $\Gamma$-action on the coordinate ring $K[G]$, and we denote
$G \aquot \Gamma = \Spec(K[G]^\Gamma)$. Then $G \aquot \Gamma$ has a natural morphism to 
$\Spec(\kk(Y))$ given by $\kk(Y) = K^\Gamma \subseteq K[G]^\Gamma$.
We get a natural morphism $G \to (G \aquot \Gamma)_K$ 
induced by $\kk[G]^\Gamma \otimes_{\kk(Y)} K \to K[G]$. We say that $G \aquot \Gamma$ is 
a \emph{$\kk(Y)$-from of $G$} if
$G \to (G \aquot \Gamma)_K$ is a  $K$-isomorphism and the 
group multiplication $G \times G \to G$ and the inverse morphism $G \to G$
descend to $\kk(Y)$-morphisms 
$G \aquot \Gamma \times_{\kk(Y)} G \aquot \Gamma \to G \aquot \Gamma$ and
$G \aquot \Gamma \to G \aquot \Gamma$.

\begin{proposition}
	\label{Prop.Descent_aff_alg_group}
	Let $G \subseteq \Bir_K(X_K)$ be an affine 
	algebraic subgroup that is preserved under $\Gamma$. 
	Then, $G \aquot \Gamma$ is a $\kk(Y)$-form of $G$.
\end{proposition}

\begin{proof}
	Denote by $\theta \colon G \times X_K \dashrightarrow G \times X_K$ the algebraic family 
	associated to $G \subseteq \Bir_K(X_K)$. 
	Using the Weil regularization theorem (see e.g.~\cite{Br2022On-models-of-algeb})
	there exists a projective variety $Y \subseteq \PP^n$ together with a faithful regular 
	right $G$-action, and a $K$-birational $G$-equivariant map $f \colon X_K \dashrightarrow Y$.
	We may choose  $x_1, \ldots, x_n \in \lociso(f) \cap \pr_2(\lociso(\theta))$, where $\pr_2$
	is the projection onto $X_K$, such that 
	\[
		\iota \colon G \dashrightarrow (X_K)^n \, , \quad
		g \dashmapsto (g(x_1), \ldots, g(x_n))
	\]
	is a rational map that induces a birational map onto its image,
	see~\cite[Lemma~5.1]{ReUrSa2024The-Structure-of-A}.
	Note that the closure of the image of $\iota \colon G \dashrightarrow (X_K)^n$ intersects 
	$G x_1 \times \cdots \times G x_n$.
	Moreover, $G x_1 \times \cdots \times G x_n \subseteq (X_K)^n$ is contained
	in $\lociso(f)^n$, as $x_1, \ldots, x_n \in \lociso(f)$
	and $\lociso(f)$ is $G$-stable. Hence, the composition 
	$\jmath \coloneqq f^n \circ \iota \colon G \dashrightarrow Y^n$ is well-defined
	and equal to $g \dashmapsto (g(f(x_1)), \ldots, g(f(x_n)))$. Since 
	$G$ acts regularly on $Y$, the map $\jmath \colon G \to Y^n$ is a morphism
	that is birational onto the closure of its image. This implies that the 
	stabilizer of $(f(x_1), \ldots, f(x_n))$ in $G$ is trivial, i.e.~$\jmath$ is injective.
	Hence, $\jmath$ defines an isomorphism of $G$ onto the $G$-orbit
	$\jmath(G) \subseteq Y^n$.

	We may choose an intermediate field $\kk(Y) \subseteq L \subseteq K$ such that $\kk(Y) \subseteq L$
	is a finite Galois extension and the following holds: the points
	$x_1, \ldots x_n \in X_K$ and the closed subset $Y \subseteq \PP^n$ are defined over $L$,
	and $f \colon X_K \dashrightarrow Y$
	is defined over $L$ as well. 
	By using~\eqref{Eq.Alg_subgroup_of_Bir_preserved_by_Gamma}, 
	for all $\tau \in \Aut_{L}(K) \subseteq \Gamma$ the following diagram
	\begin{equation}
	\label{Eq.Comm.diagram.Descent}
	\begin{gathered}
	\xymatrix@R=15pt{
		G \ar[d]_-{\tau_G} \ar@/^1pc/[rrrr]^-{\jmath}
		\ar@{-->}[rr]_-{\iota} && (X_K)^n 
		\ar[d]^-{(\tau_{X_K})^n} \ar@{-->}[rr]_-{f^n} && Y^n \ar[d]^{(\tau_Y)^n} \\
		G \ar@/_1pc/[rrrr]_-{\jmath} 
		\ar@{-->}[rr]^-{\iota} && (X_K)^n \ar@{-->}[rr]^-{f^n} && Y^n
	}
	\end{gathered}
	\end{equation}
	commutes. The commutativity of~\eqref{Eq.Comm.diagram.Descent} 
	implies that $(\tau_Y)^n$ preserves
	$\jmath(G)$ and its closure $\overline{\jmath(G)}$ for all $\tau \in \Aut_L(K)$. 
	Hence, $\overline{\jmath(G)}$ and $\overline{\jmath(G)} \setminus \jmath(G)$ are 
	defined over $L$ (note that $\jmath(G)$ is open in its closure), 
	see e.g.~\cite[Ch. AG, Theorem~14.4]{Bo1991Linear-algebraic-g}.
	This implies that the affine $K$-variety $\jmath(G)$ is also defined over $L$, i.e.
	\begin{equation}
		\label{Eq.Invaraint_Ring_intermediate}
		K[\jmath(G)]^{\Aut_L(K)} \otimes_L K = K[\jmath(G)] \, .
	\end{equation}
	Again, by the commutativity of~\eqref{Eq.Comm.diagram.Descent} we get
	$K[G]^{\Aut_L(K)} \otimes_L K = K[G]$. Since $\Aut_{\kk(Y)}(L)$
	is finite, we get now by \cite[Ch. AG, \S14.2]{Bo1991Linear-algebraic-g} that
	\[
		(K[G]^{\Aut_L(K)})^{\Aut_{\kk(Y)}(L)} \otimes_{\kk(Y)} L = K[G]^{\Aut_L(K)} \, ,
	\]
	which implies together with~\eqref{Eq.Invaraint_Ring_intermediate} 
	that $K[G]^\Gamma \otimes_{\kk(Y)} K = K[G]$. Now, the statement follows from the fact that 
	$\Gamma$ acts through group automorphisms on $G$.
\end{proof}

\begin{proposition}
	\label{Prop.Descent_Ga}
	Let $\sigma \in \Bir(X/Y)$. Then 
	$\sigma_K \in \Bir_K(X_K)$ is unipotent if and only if $\sigma \in \Bir(X/Y)$ is unipotent.
\end{proposition}

Note that Proposition~\ref{Prop.Descent_Ga} does not hold 
if we replace $\GG_a$ by $\GG_m$.
Indeed, consider e.g.~$X = \AA^2$, $Y = \AA^1$,
$\pi \colon \AA^2 \to \AA^1$, $(x, y) \mapsto y$ and let
$\sigma \colon \AA^2 \to \AA^2$ be
defined by $\sigma(x, y) = (yx, y)$. Then 
$X_K = \AA^1_{K}$, $\sigma_K$
is given by $x \mapsto yx$, and it is contained in the image of 
$\rho_{\alpha} \colon \GG_{m, K} \to \Bir_K(\AA^1_K)$, where $\alpha$ is
the $\GG_{m, K}$-action on
$\AA^1_K$ given by 
\[
	\begin{array}{rcl}
		\GG_{m, K} \times_{K} \AA^1_K &\to& \AA^1_K \\
		(t, x) &\mapsto& tx \, .
	\end{array}
\]
However, $\sigma$ cannot be in the image of a rational algebraic group-action on $\AA^2$,
since the subgroup generated by $\sigma$ in $\Bir(\AA^2)$ is not of bounded degree
(see e.g.~\cite[Lemma~2.19]{BlFu2013Topologies-and-str}).

\medskip

For the proof of Proposition~\ref{Prop.Descent_Ga} we use an algebraic characterization
of rational $\GG_a$-actions that we describe now.
Let $L$ be a field of characteristic zero (not necessarily algebraically closed) 
and let $Z$ be an irreducible variety over $L$
with function field $L(Z)$.
An $L$-derivation $\partial \colon L(Z) \to L(Z)$ is called \emph{rationally integrable},
if the formal exponential homomorphism
\[
	\exp(t \partial) \colon L(Z) \to L(Z)[[t]] \, , 
	\quad f \mapsto \sum_{i = 0}^\infty \frac{\partial^i(f)}{i!}t^i
\]
factors trough $L(Z)(t) \cap L(Z)[[t]]$.
The map
\[
	\begin{array}{rcl}
		\left\{ \,
			\begin{array}{l} 
			 \textrm{rationally integrable}  \\
			 \textrm{$L$-derivations on $L(Z)$} 
			\end{array}
		\right\}	 & \to &
		\{ \,  \textrm{rational $\GG_{a, L}$-actions on $Z$} \, \} \\
		\partial & \mapsto & \alpha_{\partial}
	\end{array}
\]
is a bijection, where the comorphism of  
$\alpha_{\partial} \colon \GG_{a, L} \times_L Z \dasharrow Z$ is 
the $L$-field homomorphism
$\exp(t \partial) \colon L(Z) \to L(Z)(t)$, 
see~\cite[Theorem~1.5]{DuLi2016Rationally-integra}.

\begin{proof}[Proof of Proposition~\ref{Prop.Descent_Ga}]
	Without loss of generality, $\sigma$ is different from the identity.

	Assume first that $\sigma_K$ is unipotent. Let
	$\alpha$ be a rational $\GG_{a, K}$-action on $X_K$ with 
	$\rho_{\alpha}(1) = \sigma_K$ and 
	let $\partial \colon K(X) \to K(X)$ be the $K$-derivation associated to 
	$\alpha$, where $K(X) = K(X_K)$ is the function field of $X_K$.
	Since $K \subseteq K(X)$ 
	is a finitely generated field extension, 
	there exists an intermediate field $\kk(Y) \subseteq L \subseteq K$
	such that $\kk(Y) \subseteq L$ is a finite Galois extension
	and $\exp(t\partial)$ restricts to an $L$-algebra homomorphism 
	$L(X) \to L(X)(t)$, where $L(X) = L(X_L)$.
	As $L(X)$ is generated by the subfields $\kk(X)$ and $L$ and as 
	$\kk(Y) \subseteq L$ is a finite Galois extension, it follows that 
	$\kk(X) = \kk(Y)(X_{\kk(Y)}) \subseteq L(X)$ is a finite Galois extension as well 
	(see~e.g. \cite[Ch. VI, \S1, Theorem~1.12]{La2002Algebra}).
	
	Let $f \in \kk(X)$. Then 
	\[
		\exp(t \partial)(f) = \frac{p}{q} \in L(X)(t) \, ,
	\]
	for $p, q \in L(X)[t]$ where $q$ is non-zero. Then there exists $n_0 \geq 0$ such that
	for all $n \geq n_0$ we have $q(n) \neq 0$.
	As $\rho_{\alpha}(1)$ is defined over $\kk(Y)$
	it follows that for all $n \geq 1$ the $n$-fold composition $\rho_{\theta}(n)$ 
	is defined over $\kk(Y)$ as well. Hence, $\frac{p(n)}{q(n)} \in \kk(X)$ for all $n \geq n_0$. 
	Let $\Gamma$ be the Galois group of $\kk(X) \subseteq L(X)$.
	For $\gamma \in \Gamma$
	denote by $q^\gamma \in L(X)[t]$ the polynomial, where $\gamma$ is applied to each coefficient
	of $q$. By multiplying the nominator and denominator of 
	the fraction $\frac{p}{q}$ with the product 
	$\prod_{\gamma \in \Gamma \setminus \{ \id \} } q^{\gamma}$
	we may assume that $q \in \kk(X)[t]$. Thus, we get $p(n) \in \kk(X)$ for all $n \geq n_0$.
	Writing $p = \sum_{j=0}^m p_j t^j$ for $p_j \in L(X)$, $j = 0, \ldots, m$, yields
	\[
		\begin{pmatrix}
			1 & n & \cdots & n^m \\
			1 & n+1 & \cdots & (n+1)^m \\
			\vdots & \vdots &  & \vdots \\
			1 & n+m & \cdots & (n+m)^m 
		\end{pmatrix}
		\begin{pmatrix}
			p_0 \\ p_1 \\ \vdots \\ p_m
		\end{pmatrix} \in \kk(X)^{m+1} \, .
	\]
	Using that the Vandermonde matrix above is invertible over $\QQ$
	we get $p_0, \ldots, p_m \in \kk(X)$, i.e.~$p \in \kk(X)[t]$. In summary,
	we conclude that
	$\exp(t \partial)$ restricts to a $\kk(Y)$-algebra homomorphism $\kk(X) \to \kk(X)[[t]]$ 
	that factors through 
	\[
		\kk(X)(t) \cap \kk(X)[[t]] = \kk(X)(t) \cap K(X)[[t]] \, .
	\]
	Moreover, 
	$\partial(f) = \frac{\textrm{d}}{\textrm{d} t} \exp(t \partial)(f) |_{t = 0} \in \kk(X)$
	for all $f \in \kk(X)$ and thus $\partial$ restricts to a $\kk$-derivation of $\kk(X)$
	that is rationally integrable. The corresponding rational $\GG_a$-action $\beta$
	on $X$ satisfies $\rho_{\beta}(1) = \sigma$ and thus $\sigma$ is unipotent.

	Assume now, that $\sigma$ is unipotent. Then there
	exists a rational $\GG_a$-action $\gamma \colon \GG_a \times X \dashrightarrow X$ such that
	$\rho_{\gamma}(1) = \sigma$ and
	$\pi \circ \gamma = \pi \circ \pr_X$, where $\pr_X \colon \GG_a \times X \to X$
	denotes the projection to $X$. Then 
	\[
		\GG_{a, K} \times_K X_K = \GG_a \times  X_K \xdashrightarrow[]{\gamma_K} X_K
	\]
	is a rational $\GG_{a, K}$-action on $X_K$ that acts as $\sigma_K$ at time $t = 1$, where
	$\gamma_K$ denotes the pull-back of $\gamma$ via $\Spec(K) \to Y$.
\end{proof}

\begin{corollary}
	\label{Cor.Existence_unipotent}
	Let $G \subseteq \Bir_K(X_K)$ be a commutative algebraic subgroup 
	that contains unipotent elements different from the identity. If 
	$G$ is preserved by $\Gamma$, then there exists a unipotent
 	$\sigma \in \Bir(X/Y)$ different from the identity with $\sigma_K \in G$.
\end{corollary}

\begin{proof}
	The unipotent elements of $G$ form an affine algebraic subgroup $H$ that is preserved
	under $\Gamma$. By Proposition~\ref{Prop.Descent_aff_alg_group} it follows
	that $H \aquot \Gamma$ is a $\kk(Y)$-form of $H$.
	Let $h \in H$ be a 
	$\kk(Y)$-rational point. Since $h$ is fixed by $\Gamma$, 
	there exists $\sigma(h) \in \Bir(X/Y)$  with 
	$\sigma(h)_K = h$, see~\eqref{Eq.Alg_subgroup_of_Bir_preserved_by_Gamma}. 
	Let $d = \dim H$.
	Note that $H \simeq (\GG_{a, K})^d$ and that $(\GG_{a, K})^d$ admits only the
	$\kk(Y)$-form $(\GG_{a, \kk(Y)})^d$, see 
	e.g.~\cite[Ch. III, \S1.1, Lemme~1 and \S1.3, Corrolaire]{Se1973Cohomologie-Galois}. Hence,
	there exists a $\kk(Y)$-rational point $\id_K \neq h \in H$.
	Now, $\sigma(h)_K = h \in H$ and $\sigma(h) \in \Bir(X/Y)$ is unipotent
	by Proposition~\ref{Prop.Descent_Ga}.
\end{proof}

\begin{corollary}
	\label{Cor.unipot_commuting}
	Assume that $\pi \colon X \to Y$ is a Rosenlicht quotient of 
	a closed, connected, commutative subgroup $H \subseteq \Bir(X)$
	and let $\varepsilon \colon \Bir(X/Y) \to \Bir_K(X_K)$ be the pull-back homomorphism.
	If $\overline{\varepsilon(H)}$ contains unipotent elements, then there exists a unipotent
	$\sigma \in \Bir(X/Y)$ different from the identity that commutes with $H$.
\end{corollary}

\begin{proof}
	By Corollary~\ref{Cor.Centralizer_pull-back}, $\overline{\varepsilon(H)}$ is a commutative
	algebraic subgroup of $\Bir_K(X_K)$. 
	Now, by Corollary~\ref{Cor.Existence_unipotent} there exists a unipotent
	$\sigma \in \Bir(X/Y)$ with $\varepsilon(\sigma) \in \overline{\varepsilon(H)}$.
	As $\varepsilon$ is injective, the statement follows.
\end{proof}

\section{\texorpdfstring{Connected solvable subgroups of $\Bir(X)$}{Connected solvable subgroups of Bir(X)}}
\label{sec.Borel}

For the whole section, $\kk$ is an uncountable, 
algebraically closed field of characteristic zero.
The goal of this section is the proof of our main result about Borel subgroups of $\Bir(X)$,
Theorem~\ref{thm:mainborel}, and to deduce several consequences of it, namely 
Corollaries~\ref{Cor.Solvable_contained_in_Borel},
\ref{Cor.Standard_Borel},
\ref{Cor.popov},
\ref{Cor.FurterPoloni}.

\medskip

Recall that for a group $G$, its \emph{$i$-th derived subgroup} $\D^i(G)$ is defined inductively
as follows: $\D^0(G) = G$ and for $i \geq 0$, $\D^{i+1}(G)$ is the commutator subgroup 
$[\D^i(G), \D^i(G)]$ of $\D^i(G)$. We denote by $\len(G) \in \NN_0 \cup \{ \infty \}$ 
its derived length,~i.e.~
the smallest $i$ such that $\D^i(G)$ is trivial. A group is called \emph{solvable},
if $\len(G)$ is finite.
Note that the non-trivial commutative groups 
are the solvable groups of derived length one.

For any subgroup $G$ of $\Bir(X)$ we have $\D^{i}(\overline{G}) \subseteq \overline{\D^{i}(G)}$ for all $i \geq 0$.
If, furthermore, $G$ is connected, then 
$\D^{i}(G)$ is connected as well for all $i \geq 0$.
In particular, if $G$ is a connected solvable subgroup of
$\Bir(X)$ of derived length $k$, the same holds for its closure $\overline{G}$.
A closed connected solvable subgroup of $\Bir(X)$ that is maximal under
subgroups with these properties is called a \emph{Borel subgroup} of $\Bir(X)$.

\begin{example}
	\label{Exa.Standard-Borel}
	Let $n \geq 1$. Recall that $B_n \subseteq \Bir(\AA^n)$ is the 
	subgroup of all birational transformations of the form
	\[
		\begin{array}{rcl}
		\AA^n &\dasharrow& \AA^n \, , \\
		(x_1, \ldots, x_n) &\dashmapsto& 
		(c_1 x_1 + d_1, \ldots, 
		c_n(x_1, \ldots, x_{n-1}) x_n + d_n(x_1, \ldots, x_{n-1}))
		\end{array}
	\]
	where $c_i, d_i \in \kk(x_1, \ldots, x_{i-1})$ and $c_i \neq 0$ for all $i=1, \ldots, n$.
	Then, $B_n$ is a Borel subgroup of $\Bir(\AA^n)$
	by \cite[Theorem~1]{Po2017Borel-subgroups-of}. 
	Moreover, $\len(B_n) = 2n$
	by \cite[Proposition~1.4]{FuHe2023Borel-subgroups-of}.
	We call $B_n$ the \emph{standard Borel subgroup} of $\Bir(\AA^n)$.
\end{example}

In the proof of Theorem~\ref{thm:mainborel}, we need more information about the Borel subgroups
in case $n = 1$ over a not necessarily algebraically closed field. For that purpose, the 
following example is crucial:

\begin{example}
	\label{Exa.Borel_Dim_1}
	Let $Y$ be an irreducible variety  and denote by
	$\pi \colon  Y \times \AA^1\to Y$ the natural projection.
	Let $B_1(Y)$ be the subgroup of $\Bir(Y \times \AA^1/Y)$
	of all birational transformations of the form
	\[
		\begin{array}{rcl}
			Y \times \AA^1 &\dasharrow& Y \times \AA^1 \, , \\
			(y, x) &\dashmapsto& 
			(y, c(y) x + d(y))
		\end{array}
	\]
	where $c, d \in \kk(Y)$ and $c \neq 0$. For a non-square 
	$f \in \kk(Y)$, let
	$T_f(Y)$ be the subgroup of $\Bir(Y \times \AA^1/Y)$ of all birational transformations 
	of the form
	\begin{equation}
		\label{Eq.Tf}
		\begin{array}{rcl}
			\varphi_{a, b} \colon Y \times \AA^1 &\dasharrow& Y \times \AA^1 \, , \\
			(y, x) &\dashmapsto& 
			\left(y, \frac{a(y) x + b(y) f(y)}{b(y)x + a(y)} \right)
		\end{array}
	\end{equation}
	where $a, b \in \kk(Y)$ and $(a, b) \neq (0, 0)$. We will show that 
	every connected solvable subgroup of $\Bir(Y \times \AA^1/Y)$
	is conjugate to a subgroup of $B_1(Y)$ or $T_f(Y)$ for a non-square $f \in \kk(Y)$
	and $B_1(Y)$, $T_f(Y)$ are Borel subgroups of $\Bir(Y \times \AA^1/Y)$.
	Note that $\len(B_1(Y) )= 2$ and that $T_f(Y)$ is commutative.
	
	Indeed, denote by $K$ 
	an algebraic closure of $\kk(Y)$. The continuous injective
	group homomorphism
	\[
		\varepsilon \colon \Bir(Y \times \AA^1/Y) \to \Bir_K(\AA^1_K) = \PGL_2(K)
	\]
	given by pull-back has $\PGL_2(\kk(Y))$ as its image. The claim now follows from the fact
	that every connected solvable subgroup of $\PGL_2(\kk(Y))$ is contained 
	in  $\varepsilon(B_1(Y))$ or $\varepsilon(T_f(Y))$ for a non-square $f \in \kk(Y)$
	and $\varepsilon(B_1(Y))$, $\varepsilon(T_f(Y))$ are Borel subgroups of $\PGL_2(\kk(Y))$,
	see~\cite[Theorem~5.9]{FuHe2023Borel-subgroups-of}.
\end{example}

In the sequel, we consider, for an irreducible variety $Y$, 
the semi-direct product $\Bir(Y\times\AA^1/Y) \rtimes \Bir(Y)$
such that 
\[
		\Bir(Y\times\AA^1/Y) \rtimes \Bir(Y) \to \Bir(Y \times \AA^1,\pi) \, , \quad 
		(\varphi, \sigma) \mapsto \varphi \circ (\sigma \times \id_{\AA^1})
\]
is a group isomorphism, where $\pi \colon Y \times \AA^1 \to Y$ is the natural projection.
In order to prove that a subgroup is conjugate to a subgroup of the standard Borel subgroup,
the following lemma turns out to be very useful:

\begin{lemma}
	\label{Lem.Criterion_inside_Bn}
	Let $Y$ be an irreducible variety,
	$G \subseteq \Bir(Y \times \AA^1,\pi)$ a closed connected 
	solvable subgroup, and
	let $H = G \cap \Bir(Y \times \AA^1/Y)$. If
	\begin{enumerate}[left=0pt]
		\item \label{Lem.Criterion_inside_Bn1} 
		$H^\circ$ is non-trivial and contained in $B_1(Y)$ or
		\item \label{Lem.Criterion_inside_Bn2}
		$\len(H^\circ) = 2$, 
	\end{enumerate}
	then $\varphi \circ G \circ \varphi^{-1}$ is contained in $B_1(Y) \rtimes \Bir(Y)$
	for some $\varphi \in \Bir(Y \times \AA^1/Y)$.
\end{lemma}

\begin{proof}
	\eqref{Lem.Criterion_inside_Bn1}:
	Let $\Phi \colon G \to \Bir(Y)$ be the canonical group homomorphism, which is continuous.
	The following map
	\begin{equation}
		\label{Eq.Map_to_PGL2}
		\eta \colon G \xrightarrow{g \mapsto g \circ (\Phi(g) \times \id_{\AA^1})^{-1}} 
		\Bir(Y \times \AA^1/Y) \xrightarrow{\varepsilon} \PGL_2(\kk(Y))
	\end{equation}
	is continuous, where $\varepsilon$ denotes the pull-back homomorphism. Note that 
	$E \coloneqq \overline{\varepsilon(H^\circ)}$ is a closed connected solvable subgroup of
	positive dimension
	of $\PGL_2(\kk(Y))$. Up to conjugating $G$ by an element of 
	$B_1(Y)$, we may assume that $E$ is either the subgroup of diagonal, strictly upper
	triangular or upper triangular matrices in $\PGL_2(\kk(Y))$. 
	In particular, the natural action of $\Bir(Y)$ by conjugation on $\PGL_2(\kk(Y))$
	preserves $E$. Since $G$ normalizes $H^\circ$ we get thus
	\[
		\eta(g)
		E \eta(g)^{-1} = E
		\quad \textrm{for all $g \in G$} \, .
	\]
	Let $N$ be the connected component of the normalizer of $E$
	in $\PGL_2(\kk(Y))$. Note that $N$ lies in the upper
	triangular matrices. As $G$ is connected
	and as~\eqref{Eq.Map_to_PGL2} is continuous we get that $\eta(G)$ is contained in $N$.
	This implies that $G \subseteq B_1(Y) \rtimes \Bir(Y)$.

	\eqref{Lem.Criterion_inside_Bn2}: 
	By Example~\ref{Exa.Borel_Dim_1}, $H^\circ$ is conjugate inside $\Bir(Y \times \AA^1/Y)$ 
	to a subgroup of $B_1(Y)$. Hence, the statement is a consequence 
	of~\eqref{Lem.Criterion_inside_Bn1}.
\end{proof}

In the following, we list further ingredients for the proof of Theorem~\ref{thm:mainborel}.
Here, we use in an essential way that the ground field $\kk$ is uncountable. 

\begin{lemma}
	\label{Lem.uncountable_index}
	Let $G \subseteq \Bir(X)$ be a closed connected subgroup and let $H \subsetneq G$ be a closed proper subgroup. Then $[G: H]$ is uncountable.
\end{lemma}

\begin{proof}
	Assume that $[G: H]$ is countable. Take $h \in H$ and $g \in G \setminus H$.
	As $G$ is connected and closed in $\Bir(X)$, there exists
	a morphism $\rho \colon V \to \Bir(X)$
	from an irreducible variety $V$
	such that its image contains $h$ and $g$.
	The preimage under $\rho$ of all $H$-cosets in $G$ gives us a decomposition
	of $V$ into countably many closed subsets, where at least two of them are non-empty.
	Since $V$ is irreducible and $\kk$ is uncountable, we obtain a contradiction 
	(e.g. by \cite[Lemma~1.3.1]{FuKr2018On-the-geometry-of}).
\end{proof}

\begin{lemma}
	\label{Lem.Reduction_to_countable_index_subgrp}
	Let $G \subseteq \Bir(X)$ be a closed connected subgroup and 
	$H \subseteq G$ a countable
	index subgroup. If $G$ is solvable, then $\len(G) = \len(H)$.
\end{lemma}

\begin{proof}[{Proof of Lemma~\ref{Lem.Reduction_to_countable_index_subgrp}}]
	The closure $\overline{H}$ of $H$ in $G$ has to be equal to $G$ by Lemma~\ref{Lem.uncountable_index} and thus $\len(G) = \len(H)$.
\end{proof}

\begin{lemma}
	\label{Lem.Reduction_to_countable_quotient}
	Let $G \subseteq \Bir(X)$ be a closed connected subgroup and 
	$N \subseteq G$ a closed normal subgroup. 
	If $G/N^\circ$ is solvable, then $\len(G/N) = \len(G/N^{\circ})$.
\end{lemma}

\begin{proof}
	For the proof we endow $G/N$ and $G/N^{\circ}$ 
	with the quotient topology given by the quotient maps
	$\pi \colon G \to G/N$ and $\pi_0 \colon G \to G/ N^{\circ}$,
	respectively. We may assume $G \neq N$ and hence,
	$k \coloneqq \len(G/N^{\circ}) > 0$. Then, by the continuity of $\pi_0$ we get
	\[
		\{1\} \neq \D^{k-1}(G/N^{\circ}) \subseteq \pi_0(\overline{\D^{k-1}(G)}) 
		\subseteq  \overline{\D^{k-1}(G/N^{\circ})} \, .
	\]
	Since $\pi_0(\overline{\D^{k-1}(G)})$ is non-trivial, we get that
	$N^{\circ} \cap \overline{\D^{k-1}(G)}$ is a proper, closed subgroup of $\overline{\D^{k-1}(G)}$.
	By Lemma~\ref{Lem.uncountable_index} we get that $\pi_0(\overline{\D^{k-1}(G)})$
	and hence $\overline{\D^{k-1}(G/N^{\circ})}$ is uncountable. 
	As $N^\circ$ has countable index in $N$,
	it follows that the restriction of the continuous group homomorphism 
	$G/N^\circ \to G/N$ to  
	$\overline{\D^{k-1}(G/N^{\circ})} \to \overline{\D^{k-1}(G/N)}$ has a countable kernel. In particular, $\overline{\D^{k-1}(G/N)}$ is uncountable, hence non-trivial.
	This shows that $\D^{k-1}(G/N)$ is non-trivial and thus
	\[
		k \leq \len(G/N) \leq \len(G/N^{\circ}) = k \, . \qedhere
	\]
\end{proof}

\begin{lemma}
	\label{Lem.Autos_of_alg_grps}
	Let $H$ be a commutative connected algebraic group, let $G$ be a closed connected solvable subgroup of $\Bir(X)$, and let $\iota \colon G \to \Aut_{\textrm{alg.grp}}(H)$ be a group homomorphism.
	Then $\len(G) \leq \len( \ker(\iota)) +  \dim H_u$, where $H_u$ denotes the 
	subgroup of unipotent elements of $H$.
\end{lemma}

\begin{proof}
	By Lemma~\ref{Lem.Structure_Aut_alg.grp} there is a countable index subgroup 
	$L \subseteq \Aut_{\textrm{alg.grp}}(H)$ that we will consider as a subgroup of 
	$\GL_n$, where $n = \dim H_u$.
	Let $F$ be the intersection of $\iota(G) \cap L$ with the connected
	solvable algebraic subgroup $\overline{\iota(G) \cap L}^\circ$ of $\GL_n$ inside $\GL_n$. 
	Then $F$ is a finite index subgroup
	of $\iota(G) \cap L$ and $F$ is contained in a Borel subgroup of $\GL_n$.
	As $F$ has countable index in $\iota(G)$
	the same holds for $\iota^{-1}(F)$ inside $G$. 
	Now we get the following estimate:
	\[
		\len(G) \xlongequal{\textrm{Lem.~\ref{Lem.Reduction_to_countable_index_subgrp}}} 
		\len(\iota^{-1}(F)) \leq \len(\ker(\iota)) + \len(F) \, .
	\]
	As $F$ is contained in a Borel subgroup of $\GL_n$, we get $\len(F) \leq n$.
\end{proof}

\begin{remark}
	If $\dim H_u \geq 2$, then the bound in Lemma~\ref{Lem.Autos_of_alg_grps} can be improved to
	$\len(G) \leq \len(\ker(\iota)) + \lceil\log_2(\dim H_u -1)\rceil + 2$, 
	see e.g.~\cite[\S18, Theorem~1]{Su1976Matrix-groups}.
\end{remark}

\begin{proof}[{Proof of Theorem~\ref{thm:mainborel}}]
	Without loss of generality we assume that $G$ is closed in $\Bir(X)$.
	We proceed by induction on $n = \dim X$. If $n = 1$, then we may assume that 
	$X$ is a smooth projective irreducible curve. As $\Aut(X)$ is an algebraic group
	acting regularly on $X$,
	the natural group isomorphism $\Aut(X) \to \Bir(X)$ is a homeomorphism.
	Now, the result follows for $n = 1$, since $\Aut(X)^\circ$ is either $\PGL_2(\kk)$, 
	an elliptic curve or trivial.
	Assume that $n > 1$. Let $k = \len(G)$. Then $G_{k-1} \coloneqq \overline{\D^{k-1}(G)}$ is a closed connected
	commutative subgroup of $\Bir(X)$. We may assume that $G$ is non-trivial and thus
	$G_{k-1}$ is non-trivial as well.
	We distinguish whether $\kk(X)^{G_{k-1}}$ is equal to $\kk$
	or not.

	Assume first that $\kk(X)^{G_{k-1}} = \kk$.
	By Corollary~\ref{Cor.Centralizer} it follows that $G_{k-1} \subseteq \Bir(X)$ is a 
	self-centralizing, connected, algebraic subgroup of dimension $n$.
	Consider the following group homomorphism
	\[
		\iota \colon G \to \Aut_{\textrm{alg.grp}}(G_{k-1}) \, , \quad 
		g \mapsto (h \mapsto g \circ h \circ g^{-1})	\, .
	\]
	The kernel of $\iota$ is equal to $G_{k-1}$ (as $G_{k-1}$ is self-centralising)
	and hence we get $\len(G) \leq 1 + n \leq 2n$ 
	by Lemma~\ref{Lem.Autos_of_alg_grps}.

	For the rest of the proof we assume $\kk(X)^{G_{k-1}} \neq \kk$. 
	Let $\pi_0 \colon X_0 \to Y$ be a Rosenlicht quotient for $G_{k-1}$, 
	see Theorem~\ref{Thm.Rosenlicht}. Since $G$ preserves the invariant field
	$\kk(X)^{G_{k-1}} = \kk(Y)$, the natural
	group homomorphism $\Bir(X_0, \pi_0) \to \Bir(Y)$ restricts to a
	continuous group homomorphism
	\[
		\varphi\colon G \to \Bir(Y) \, .
	\]
	Since $G$ is connected the image $\im(\varphi)$ is connected as well. 
	Denote by $H \subseteq G$ the kernel of $\varphi$.
	As $H$ is a closed normal subgroup of $G$, we get
$\len(\im \varphi) = \len(G/H) = \len(G/H^\circ)$ by Lemma~\ref{Lem.Reduction_to_countable_quotient}.
Since $G_{k-1}$ is non-trivial, it follows that the general fibres of
$\pi \colon X_0 \to Y$ are of positive dimension, i.e.~$\dim Y < n$.
By induction hypothesis, $\len(\im \varphi) \leq 2 \dim Y$.

Let $K$ be an algebraic closure of $\kk(Y)$. 
Consider the injective continuous group homomorphism
\[
	\varepsilon \colon \Bir(X_0/Y) \to \Bir_K(X_{0, K}) \, .
\] 
Then, $\varepsilon(H^\circ)$ is a connected solvable subgroup of 
$\Bir_K(X_{0, K})$. 
Note that we have $\dim(X_{0, K}) = n - \dim Y < n$, as $\kk(X)^{G_{k-1}} \neq \kk$.
By induction hypothesis, we get $\len \varepsilon(H^\circ) \leq 2 (n - \dim Y)$.
Since $\varepsilon$ is injective, we get thus $\len(H^\circ) \leq 2 (n - \dim Y)$.
In summary,
\begin{equation}
	\label{Eq.Estimate_length}
	\len(G) \leq \len( H^\circ ) + \len(G/H^\circ)
			\leq 2(n - \dim Y) + 2 \dim Y = 2 n \, . 
\end{equation}

It remains to assume $\len(G) = 2n$ (and still 
$\kk(X)^{G_{k-1}} \neq \kk$) and to show that $X$ is rational, and that
$G$ is conjugate to a subgroup of the standard Borel subgroup $B_n$ in $\Bir(\AA^n)$.
Let $d = n - \dim Y$.
Then $0 < d < n$.
Since $\len(G/ H^\circ) = \len(G/ H) \leq 2 \dim Y$
and $\len(H^\circ) \leq 2 d$, we get
\[
	\len(G/ H) = 2 \dim Y \quad \textrm{and} \quad \len(H^\circ) = 2 d
\]
using~\eqref{Eq.Estimate_length}.
Hence, $\overline{\varepsilon(H^\circ)}$ is a closed connected solvable subgroup of 
$\Bir_K(X_{0, K})$ of derived length  $2d$.
Corollary~\ref{Cor.Centralizer_pull-back} 
implies that $\overline{\varepsilon(G_{2n-1})}$ is a
normal, self-centralizing, commutative, connected, algebraic
subgroup of $\overline{\varepsilon(H^\circ)}$ of 
dimension $d$.
Lemma~\ref{Lem.Autos_of_alg_grps} applied to the homomorphism
$\overline{\varepsilon(H^\circ)} \to \Aut_{\textrm{alg.grp}}(\overline{\varepsilon(G_{2n-1})})$
given by conjugation yields 
\[
	2d = \len(\overline{\varepsilon(H^\circ)}) \leq 1 + d \, .
\]
Hence, $d = 1$. Since $\len(\overline{\varepsilon(H^\circ)}) = 2$ 
and $\overline{\varepsilon(G_{2n-1})}$
is self-centralizing, the group $\Aut_{\textrm{alg.grp}}(\overline{\varepsilon(G_{2n-1})})$
is uncountable by Lemma~\ref{Lem.uncountable_index}. 
Hence, $\overline{\varepsilon(G_{2n-1})}$ contains unipotent elements
(by Lemma~\ref{Lem.Structure_Aut_alg.grp}). As $d = 1$, we get
$\overline{\varepsilon(G_{2n-1})} \simeq \GG_{a, K}$. Now, Proposition~\ref{Prop.Descent_Ga} 
implies that $G_{2n-1}$ consists only of unipotent elements. 

Let $U$ be a closed one-dimensional 
algebraic subgroup of $G_{2n-1}$ that is isomorphic to $\GG_a$. 
Then $\kk(X)^U = \kk(Y)$,
since $\kk(Y)$ is algebraically closed in $\kk(X)$.
By \cite[Theorem~1]{Ma1963On-algebraic-group}
there exists a birational map $h \colon X \dashrightarrow Y \times \AA^1$ such that 
$\pi_0$ corresponds to the natural projection $Y \times \AA^1 \to Y$. 
By induction hypothesis, $Y$ is birational to $\AA^{n-1}$
and $\im(\varphi)$ is conjugate to a subgroup of 
$B_{n-1} \subseteq \Bir(\AA^{n-1})$.
Thus, we may replace $\pi_0$ by the projection
$\pr_{n-1} \colon\AA^{n} \to \AA^{n-1}$ to the first $n-1$ coordinates, 
and we may assume
\[
	G \subseteq \Bir(\AA^n, \pr_{n-1})
	\quad \textrm{and} \quad
	\varphi(G) \subseteq B_{n-1} \subseteq \Bir(\AA^{n-1}) \, .
\]
Since $\len(H^\circ) = 2$, it follows from Lemma~\ref{Lem.Criterion_inside_Bn}
that $G$ is conjugate to a subgroup of 
$B_1(\AA^{n-1}) \rtimes B_{n-1} = B_n \subseteq \Bir(\AA^n)$.
\end{proof}

As a consequence we get the following result which answers \cite[Question~3.9]{FuPo2018On-the-maximality-} affirmatively:

\begin{corollary}
	\label{Cor.Solvable_contained_in_Borel}
	Let $A \in \{\Aut(X), \Bir(X)\}$ be either the group of automorphisms of $X$
	or the group of birational transformations of $X$.
	Then any connected solvable subgroup of $A$ is contained in a Borel subgroup of $A$.
\end{corollary}

\begin{proof}
	Note that the derived length of every connected solvable subgroup in $A$ is bounded
	by Theorem~\ref{thm:mainborel} and the fact that the natural injection
	$\Aut(X) \to \Bir(X)$ is continuous. For a connected and solvable subgroup 
	$G \subseteq A$ this enables us to apply Zorn's lemma to the set 
	of all connected solvable subgroups of $A$ that contain $G$. This gives
	the desired Borel subgroup in $A$.
\end{proof}

\begin{corollary}
	\label{Cor.Standard_Borel}
	Every Borel subgroup in $\Bir(\AA^n)$ of derived length $2n$ is conjugate to the standard
	Borel subgroup $B_n$. 
\end{corollary}

\begin{proof}
	This follows directly from Theorem~\ref{thm:mainborel}. 
\end{proof}

Naturally, one is lead to the question, whether for every $n \geq 2$ there exists 
a Borel subgroup of $\Bir(\AA^n)$ that is non-conjugate to $B_n$. 
The affirmative answer is in fact a conjecture of Popov (he assumes $n \geq 5$, 
see \cite[p.6]{Po2017Borel-subgroups-of}), proven by Furter and Hedén in case $n = 2$, 
see~\cite[Theorem~1.3]{FuHe2023Borel-subgroups-of}.

The idea to construct Borel subgroups of distinct length (i.e.~Corollary~\ref{Cor.popov}) 
is to show that each element 
$\varphi_{a, b} \neq \id_{\AA^n}$ of the commutative subgroup $T_f(\AA^{n-1})$ of $\Bir(\AA^n)$ 
from Example~\ref{Exa.Borel_Dim_1}
is non-conjugate to any element of $B_n$ for a suitable $f \in \kk(x_1, \ldots, x_{n-1})$.

\begin{lemma}
	\label{Lem.Fixed_locus_std_Borel_elements}
	Any hypersurface in $\PP^n$ that is contained in the closure of the fixed locus 
	of an element of $B_n \setminus \{ \id_X \}$ is ruled, 
	i.e.~birational to a product of a variety with $\AA^1$. 
\end{lemma}

\begin{proof}
	Let $\varphi \in B_n$ be given by
	\[
	\begin{array}{rcl}
		\AA^n &\dasharrow& \AA^n \, , \\
		(x_1, \ldots, x_n) &\dashmapsto& 
		(c_1 x_1 + d_1, \ldots, c_n(x_1, \ldots, x_{n-1}) x_n + d_n(x_1, \ldots, x_{n-1}))
	\end{array}
	\]
	where $c_i, d_i \in \kk(x_1, \ldots, x_{i-1})$ and $c_i \neq 0$. Write 
	\[
		c_i = \frac{\gamma_{i, 1}}{\gamma_{i, 2}} \quad \textrm{and} \quad 
		d_i = \frac{\delta_{i, 1}}{\delta_{i, 2}} \, ,
	\] 
	where 
	$\gamma_{i, j}, \delta_{i, j} \in \kk[x_1, \ldots, x_{i-1}]$ and
	$\gamma_{i, 2} \delta_{i, 2} \neq 0$.
	Consider the open non-empty subset
	\[
		U = \Bigset{(x_1, \ldots, x_n) \in \AA^n}
		{
		\begin{array}{l}	
		\gamma_{i, 2}(x_1, \ldots, x_{i-1}) \, , \ \delta_{i, 2}(x_1, \ldots, x_{i-1}) \neq 0 \\
		\textrm{for $i = 1, \ldots, n$}
		\end{array}
		}	\subseteq \AA^n \, .
	\]
	As each irreducible hypersurface inside $\AA^n \setminus U$ is ruled, 
	it is enough to show that the irreducible hypersurfaces in $U$ that are pointwise fixed 
	by the morphism $\varphi |_U \colon U \to \AA^n$ are ruled. 
	Let $H \subseteq U$ be such a hypersurface. Choose $k$
	such that $(c_k, d_k) \neq (1, 0)$. Then $H$ is contained in the hypersurface
	\[
		\Bigset{(x_1, \ldots, x_n) \in U}
		{ 
		\begin{array}{l}
			\delta_{k, 2}(x_1, \ldots, x_{k-1}) \gamma_{k, 1}(x_1, \ldots, x_{k-1})x_k \\
			+ \, \gamma_{k, 2}(x_1, \ldots, x_{k-1}) \delta_{k, 1}(x_1, \ldots, x_{k-1}) \\ 
			= \, x_k \gamma_{k, 2}(x_1, \ldots, x_{k-1})\delta_{k, 2}(x_1, \ldots, x_{k-1})
		\end{array}
		}	\, .
	\]
	Now, as the equation is non-trivial and of degree at most one in 
	$x_k$,
	every irreducible component of the 
	hypersurface above is ruled. This implies the statement. 
\end{proof}

\begin{proof}[Proof of Corollary~\ref{Cor.popov}]
	Let $F \in \kk[x_0, \ldots, x_{n-1}]$ be a 
	homogeneous, irreducible polynomial of degree $2n$ such that 
	the hypersurface in $\PP^{n-1}$ given by $F$ is smooth. 
	In particular, $f = F(1, x_1, \ldots, x_{n-1})$ 
	is not a square in $\kk(x_1, \ldots, x_{n-1})$. 
	Let $(a, b) \neq (0, 0)$ be a pair of elements inside
	$\kk(x_1, \ldots, x_{n-1})$.
	Now, for all 
	$(x_1, \ldots, x_n) \in \AA^n$ with  
	$b(x_1, \ldots, x_{n-1}) x_n + a(x_1, \ldots, x_{n-1}) \neq 0$
	and $x_n^2 = f(x_1, \ldots, x_{n-1})$ we have 
	$\varphi_{a, b}(x_1, \ldots, x_n) = (x_1, \ldots, x_n)$.
	When we see $\varphi_{a, b}$ as a birational map of $\PP^n$ via
	the open embedding 
	$\AA^n \to \PP^n$ given by $(x_1, \ldots, x_n) \mapsto [1: x_1: \cdots: x_n]$,
	then
	\[
		H \coloneqq V_{\PP^n}(x_0^{2n-2} x_{n}^{2} - F) \subseteq \PP^n
	\]
	lies in the closure of the fixed locus of $\varphi_{a, b} \colon \dom(\varphi_{a, b}) \to \PP^n$. Note that
	\[
		\begin{array}{rcl}
			H &\dashrightarrow& X \coloneqq 
			V_{\PP(1, \ldots, 1, n)}(x_n^2 - F) \subseteq \PP(1, \ldots, 1, n) \\
			{[x_0: \cdots: x_{n}]} &\dashmapsto& 
			{[x_0: x_1: \cdots : x_{n-1}: x_0^{n-1} x_n]}
		\end{array}
	\]
	is a birational map, where $\PP(1, \ldots, 1, n)$ denotes the weighted projective
	space of dimension $n$ with weights $1, \ldots, 1, n$. Moreover, $X$ is smooth
	and the projection to $x_0, \ldots, x_{n-1}$ yields a double cover
	$\pi \colon X \to \PP^{n-1}$ that is branched along the smooth hypersurface $D$ in $\PP^{n-1}$
	given by $F$.
	By the ramification formula we have for the canonical divisors $K_X$ and
	$K_{\PP^{n-1}}$
	\[
		K_X = \pi^{\ast} K_{\PP^{n-1}} + R \, ,
	\]
	where $R$ is the ramification divisor of $\pi$, see 
	e.g.~\cite[\S5.6, Theorem~5.5]{Ii1982Algebraic-geometry}. 
	As $\pi$ is a double cover
	which is tamely ramified over $D$ ($\car{\kk} = 0$),
	we get $2 R = \pi^\ast D$ and hence 
	\[
		2 K_X = 2 (\pi^{\ast} K_{\PP^{n-1}} + R)	= \pi^\ast(2 K_{\PP^{n-1}} + D) = 0 \, ,
	\]
	where the last equality follows from the fact that
	$D$ is linearly equivalent to $2n$ times a hyperplane in $\PP^{n-1}$ and 
	$K_{\PP^{n-1}}$ is $-n$ times a hyperplane in $\PP^{n-1}$. 
	Hence, the Kodaira dimension of $X$ is zero
	and thus $X$ and $H$ are not ruled.

	On the other hand, when we see an element $\varphi \in B_n$ as a birational map of 
	$\PP^n$, then any irreducible hypersurface in the closure of the fixed  locus of 
	$\varphi \colon \dom(\varphi) \to \PP^n$ is ruled 
	(see Lemma~\ref{Lem.Fixed_locus_std_Borel_elements}).

	Now, assume that there exists a birational map $\eta \colon \PP^n \dasharrow \PP^n$
	such that $\varphi_{a, b} = \eta^{-1} \circ \varphi \circ \eta$ for some $\varphi \in B_n$.
	By the weak factorization theorem (see \cite{AbKaMaWl2002Torification-and-f}) 
	we may decompose $\eta$
	into blow-ups of smooth varieties in smooth centers and its inverses. This implies that
	any hypersurface $V$ contracted by $\eta$ is ruled. 
	In particular, $\eta$ cannot contract $H$ and thus the closure
	of $\eta(\dom(\eta) \cap H)$ lies in the closure of the fixed locus of 
	$\varphi \colon \dom(\varphi) \to \PP^n$, contradiction.

	Let $B$ be a Borel subgroup of $\Bir(\AA^n)$ that contains 
	$T_f(\AA^{n-1})$ (see Corollary~\ref{Cor.Solvable_contained_in_Borel}),
	then $B$ has derived length $< 2n$, since otherwise $B$ and $B_n$
	would be conjugate in $\Bir(\AA^n)$ by Corollary~\ref{Cor.Standard_Borel}.
\end{proof}

Moreover, we can show that $\Aut(\AA^n)$ contains non-conjugate Borel subgroups
using a similar argument to the proof in Corollary~\ref{Cor.popov} provided that 
$n \geq 3$.
In $\Aut(\AA^n)$, the intersection $\Aut(\AA^n) \cap B_n$ is a Borel subgroup,
see~\cite[Corollary~1.2]{FuPo2018On-the-maximality-}. The question, whether 
there exist non-conjugate Borel subgroups in $\Aut(\AA^n)$ for $n \geq 3$ was posed 
in~\cite[Problem~3]{FuPo2018On-the-maximality-}.
In $\Aut(\AA^2)$ all Borel subgroups are conjugate, 
see~\cite{BeEsEs2016Dixmier-groups-and} and also~\cite{FuPo2018On-the-maximality-}.

\begin{proof}[Proof of Corollary~\ref{Cor.FurterPoloni}]
	For $n = 3$ there exists a $\GG_a$-action on $\AA^3$ that is not 
	conjugate to any $\GG_a$-action on $\AA^3$ with image in $\Aut(\AA^n) \cap B_n$, 
	see~\cite[p.2]{Ba1984A-nontriangular-ac}. 
	This example can easily be adapted to get such a $\GG_a$-action on $\AA^n$ for $n \geq 3$.
	For self-completeness we insert the argument here.

	Let $f = x_1 x_3 + x_2^2 + x_4^2 + x_5^2 + \ldots + x_n^2 \in \kk[x_1, \ldots, x_n]$
	and consider the $\GG_a$-action
	\[
		\begin{array}{rcl}
		\GG_a \times \AA^n &\to& \AA^n	\\
		(t, x_1, \ldots, x_n) &\mapsto& (x_1 + 2t \Delta x_2-(t\Delta)^2 x_3, 
		x_2 - t \Delta x_3, x_3, \ldots, x_n) \, ,
		\end{array}
	\]
	where $\Delta = f(x_1, \ldots, x_n)$. 
	
	Let $G \subseteq \Aut(\AA^n)$ be the image
	under the corresponding group homomorphism $\GG_a \to \Aut(\AA^n)$. Then the irreducible 
	hypersurface $H$ in $\AA^n$ given by $f$ is fixed by every
	element of $G$. Note that $H$ is the affine cone over an irreducible quadric in $\PP^{n-1}$
	and hence $H$ has exactly one singular point.

	On the other hand, every hypersurface that is fixed
	by a non-trivial element of $\Aut(\AA^n) \cap B_n$ is isomorphic to 
	$\AA^1 \times Y$ for some $(n-2)$-dimensional affine variety $Y$
	(this follows from a similar argument to the proof of 
	Lemma~\ref{Lem.Fixed_locus_std_Borel_elements}).

	Hence, no non-trivial element of $G$ is conjugate to an element in $\Aut(\AA^n) \cap B_n$.
	Using Corollary~\ref{Cor.Solvable_contained_in_Borel} we get the statement.
\end{proof}

\section{Characterization of rationality and ruledness}\label{sec.Char}

In this section $\kk$ is uncountable, algebraically closed and of characteristic zero. 
The goal is to prove the following characterization of rationality and ruledness
via its group of birational transformations, i.e.~Theorem~\ref{thm:charmain} and Theorem~\ref{thm:charmain2}.

\subsection{Existence of unipotent elements}
The proofs of Theorem~\ref{thm:charmain} and Theorem~\ref{thm:charmain2} use heavily 
the existence of unipotent elements that commute with certain subgroups of 
the group of birational transformations. This constitutes the main result of this subsection:

\begin{theorem}
	\label{Thm.Existence_unipot}
	Let $G, H \subseteq \Bir(X)$ be closed connected commutative non-trivial
	subgroups. Assume that
	$G$ normalizes $H$ and that $H$ is covered by countably many $G$-conjugacy class
	closures. 
	Then $\Bir(X)$ contains unipotent elements that are different from the identity and
	commute with $H$.
\end{theorem}

\begin{lemma}
	\label{Lem.Existence_unipot_elements}
	Let $C$ be a commutative algebraic group, $C_0 \subseteq C$
	a dense subgroup and 
	$M \subseteq \Aut_{\textrm{\textrm{alg.grp}}}(C)$ an uncountable group
	preserving $C_0$. Then $C_0$ contains unipotent elements that are different from the identity.
\end{lemma}

\begin{proof}[Proof of Lemma~\ref{Lem.Existence_unipot_elements}]
	Let $U \subseteq C$ be the closed normal subgroup of unipotent elements. 
	Then $\Aut_{\textrm{alg.grp}}(C/U)$
	is countable by Lemma~\ref{Lem.Structure_Aut_alg.grp}. 
	Hence, the subgroup $M_0 \subseteq M$ of elements that act trivially 
	on $C/U$ has countable index in $M$.
	As $M$ is uncountable, we get $M_0 \neq \{ \id_X \}$.
	Let $m_0 \in M_0$ and $c_0 \in C_0$ with $m_0(c_0) \neq c_0$.
	Then $m_0(c_0) c_0^{-1}$ is an element in $U$ that is also contained in $C_0$.
\end{proof}

\begin{lemma}
	\label{Lem.Orbit_dimension}
	Let $G, H \subseteq \Bir(X)$ be closed connected subgroups such that $G$ normalizes $H$.
	Denote by $\pi_0 \colon X_0 \to Y$ a Rosenlicht quotient for $H$
	and let $\varphi \colon G \to \Bir(Y)$ be the restriction of
	$\Bir(X_0,\pi_0) \to \Bir(Y)$ to $G$. Assume that
	\begin{itemize}[left=0pt]
		\item $\overline{\varphi(G)}$ is an algebraic subgroup of $\Bir(Y)$;
		\item $\ker(\varphi) \cap \Cent_G(H)$ has countable index in $\ker(\varphi)$. 
	\end{itemize}
	Then $\dim \overline{C_G(h)} \leq \dim \overline{\varphi(G)}$ for all $h \in H$,
	where $C_G(h)$ denotes the $G$-conjugacy class of $h$ in $H$.
\end{lemma}

\begin{proof}
 	Let 
	$\eta \colon G \to G / (\ker(\varphi) \cap \Cent_G(H))$ be the natural projection.
	We endow this quotient and $\ker(\varphi) / (\ker(\varphi) \cap \Cent_G(H))$
	with the quotient topology.
	As $\ker(\varphi)$ is closed in $G$, 
	the inclusion 
	\[
		\ker(\varphi) / (\ker(\varphi) \cap \Cent_G(H)) \hookrightarrow
		G / (\ker(\varphi) \cap \Cent_G(H))
	\]
	is a closed embedding.
	Using that $\ker(\varphi)^\circ \cap \Cent_G(H)$
	is a closed subgroup of countable index in $\ker(\varphi)^\circ$
	gives that $\ker(\varphi)^\circ$ is contained in $\Cent_G(H)$
	(see Lemma~\ref{Lem.uncountable_index}). As $\ker(\varphi)^\circ$ is open 
	in $\ker(\varphi)$, the same holds for $\ker(\varphi) \cap \Cent_G(H)$ in $\ker(\varphi)$.
	Hence, $\ker(\varphi) / (\ker(\varphi) \cap \Cent_G(H))$ is discrete and
	countable. In particular, $\varphi$ factors through a continuous map
	$\overline{\varphi} \colon G / (\ker(\varphi) \cap \Cent_G(H)) \to G'$
	with countable discrete fibres, where 
	$G' \coloneqq \overline{\varphi(G)} \subseteq \Bir(Y)$.
	
	Let $G_1 \subseteq G_2 \subseteq \cdots \subseteq \Bir(X)$ be a chain of closed 
	irreducible algebraic subsets such that their union is $G$. 
	For $d \geq 1$, there exists an open dense subset
	$U_d \subseteq G_d$ such that $U_d$ is a normal variety and the inclusion $U_d \to \Bir(X)$
	is a morphism 
	(see \cite[Proposition~5.2]{ReUrSa2024The-Structure-of-A}).
	Consider the following commutative diagram:
	\[
		\xymatrix@R=20pt{
			G_d \ar@{}[r] |-{\supseteq} & U_d \ar[rrr]^{\eta_d \coloneqq \eta |_{U_d}}
			\ar[rrrd]_-{\varphi_d \coloneqq \varphi |_{U_d}} &&& \eta(U_d) 
			\ar[d]^-{\overline{\varphi}_d \coloneqq \overline{\varphi} |_{\eta(U_d)}}
			\ar@{}[r] |-{\subseteq} & G / (\ker(\varphi) \cap \Cent_G(H)) \, ,
			\ar@/^1pc/[ld]^{\overline{\varphi}} \\
			& &&& G'
		}	
	\]
	As $U_d$ is a Noetherian topological space, the same holds for the image
	$\eta(U_d)$ and in particular, this holds for the fibres of $\overline{\varphi}_d$. 
	As the fibres of $\overline{\varphi}$ are discrete and countable,
	it follows that the fibres of $\overline{\varphi}_d$ are finite.

	Note that $U_d \xrightarrow{\varphi_d} G' \subseteq \Bir(Y)$ is a morphism to $\Bir(Y)$, 
	as $\Bir(X_0, \pi_0) \to \Bir(Y)$ preserves algebraic families 
	parametrized by normal varieties, see
	\cite[Proposition~7.5]{ReUrSa2024The-Structure-of-A}.
	For $d \geq 1$ large enough we may assume that $\varphi_d \colon U_d \to G'$ is dominant.
	By applying the fibre-dimension-formula to the irreducible variety $U_d$
	and the morphism $U_d \xrightarrow{\varphi_d} G' \subseteq \Bir(Y)$
	we may further assume after shrinking $U_d$ 
	that all non-empty fibres of $\varphi_d$
	are equidimensional of dimension $f_d \coloneqq \dim U_d - \dim G'$. 
	As the fibres of $\overline{\varphi}_d$
	are finite and discrete, it follows that all non-empty fibres of $\eta_d$ are equidimensional
	of dimension $f_d$ as well.

	We fix $h \in H$ and apply the fibre-dimension-formula 
	to the morphism
	$\rho_d \colon U_d \to \Bir(X)$ given by $g \mapsto ghg^{-1}$ in order to get 
	$\dim \overline{\rho_d(U_d)} + \dim_g \rho_d^{-1}(\rho_d(g)) = \dim U_d$
	for some $g \in U_d$. Hence, 
	\[
		\dim \overline{\rho_d(U_d)} = \dim U_d - \dim \rho_d^{-1}(\rho_d(g)) 
		\leq \dim U_d - f_d = \dim G' \, ,
	\]
	where we used that $\rho_d \colon U_d \to \Bir(X)$ factors through 
	$\eta_d \colon U_d \to \eta(U_d)$.
	Using that the dimension of a closed algebraic subset in $\Bir(X)$ is given by its 
	Krull dimension, the increasing sequence of closed irreducible algebraic subsets
	$\overline{\rho_d(U_1)} \subseteq \overline{\rho_d(U_2)} \subseteq \cdots$ of $\Bir(X)$
	becomes stationary. Hence, for $d \geq 1$ large enough $\overline{\rho_d(U_d)}$
	contains the conjugacy class $C_G(h)$ and thus $\dim \overline{C_G(h)} \leq \dim G'$.
\end{proof}

First we proof a version of Theorem~\ref{Thm.Existence_unipot} in case $GH$ acts without invariants
on $X$, and then we reduce the general case to that situation by pulling back to the geometric 
generic fibre of a Rosenlicht quotient for $GH$.

\begin{lemma}
	\label{Lem.Existence_unipot_trivial_inv}
	Let $G, H \subseteq \Bir(X)$ be closed, connected, commutative, non-trivial
	subgroups such that $G$ normalizes $H$, but does not centralize $H$. 
	Assume that:
	\begin{itemize}[left=0pt]
		\item there is an ascending chain $A_1 \subseteq A_2 \subseteq \cdots$ 
		of closed irreducible subsets in $H$ such that the union $\bigcup_{i \geq 1} A_i$
		is dense in $H$ and each $A_i$ is contained in the closure of a $G$-conjugacy
		class in $H$
		\item $\kk(X)^{GH} = \kk$
	\end{itemize}
	Then $H$ contains unipotent elements different from the identity.
\end{lemma}

\begin{proof}[Proof of Lemma~\ref{Lem.Existence_unipot_trivial_inv}]
	After shrinking $X$ we may assume that there exists
	a Rosenlicht quotient $\pi \colon X \to Y$ for $H$. 
	Let $\varphi \colon G \to \Bir(Y)$ be the restriction of $\Bir(X, \pi) \to \Bir(Y)$ 
	to $G$. We get 
	\[
		\kk(Y)^{\overline{\varphi(G)}} = \kk(Y)^{\varphi(G)} 
		= (\kk(X)^H)^{\varphi(G)} = \kk(X)^{GH} = \kk \, ,
	\] 
	where the first inequality follows from~\cite[Corollary~7.4]{ReUrSa2024The-Structure-of-A}.
	Thus, Corollary~\ref{Cor.Centralizer} yields that
	$\overline{\varphi(G)}$ is an algebraic subgroup of $\Bir(Y)$.
	Denote by $K$ an algebraic closure of $\kk(Y)$ and consider the continuous, injective 
	pull-back homomorphism
	\[
			\varepsilon \colon \Bir(X/Y) \to \Bir_K(X_K) \, ,
	\]
	where $X_K$ is the geometric generic fibre of $\pi$. 
	By Corollary~\ref{Cor.Centralizer_pull-back}, $\overline{\varepsilon(H)}$
	is an algebraic subgroup of $\Bir_K(X_K)$.

	Assume first that $\ker(\varphi) / (\ker(\varphi) \cap \Cent_G(H))$ is uncountable.
	Observe that the uncountable group 
	$\varepsilon(\ker(\varphi)) / \varepsilon(\ker(\varphi) \cap \Cent_G(H))$
	acts faithfully by conjugation on $\overline{\varepsilon(H)}$  and hence
	$\overline{\varepsilon(H)}$ contains unipotent elements
	(see Lemma~\ref{Lem.Structure_Aut_alg.grp}). Lemma~\ref{Lem.Existence_unipot_elements}
	applied to $C = \overline{\varepsilon(H)}$, $C_0 = \varepsilon(H)$
	and the image of $\varepsilon(\ker(\varphi))$ 
	in $\Aut_{\textrm{alg.grp}}(\overline{\varepsilon(H)})$ under the conjugation
	homomorphism
	yields that $\varepsilon(H)$ contains unipotent elements that are different from the 
	identity. Proposition~\ref{Prop.Descent_Ga}
	gives us the desired unipotent element in $H$.

	Form now on we assume that $\ker(\varphi) / (\ker(\varphi) \cap \Cent_G(H))$ is countable.
	By Lemma~\ref{Lem.Orbit_dimension} there exists $C \geq 0$ (independent of $h$) such that
	$\dim \overline{C_G(h)} \leq C$.
	By assumption, we get now $\dim A_i \leq C$ for all $i \geq 1$. As the $A_i$ are closed
	and irreducible in $\Bir(X)$, the sequence $(\dim A_i)_{i \geq 1}$
	becomes stationary. As $\bigcup_{i \geq 1} A_i$
	is closed and dense in $H$, we get 
	$H = A_i$ for sufficiently large $i$ and in particular $\dim H < \infty$.
	Hence, $H \subseteq \Bir(X)$ is an algebraic subgroup.
	Note that the image of $G$ in $\Aut_{\textrm{alg.grp}}(H)$ under the conjugation homomorphism
	is isomorphic to $G/\Cent_G(H)$. However, since by assumption $\Cent_G(H) \neq G$, 
	the quotient $G/\Cent_G(H)$ is uncountable, see Lemma~\ref{Lem.uncountable_index}.
	Thus, $\Aut_{\textrm{alg.grp}}(H)$ is uncountable, and by 
	Lemma~\ref{Lem.Structure_Aut_alg.grp} we get that $H$ contains unipotent elements that are different from the identity. 
\end{proof}

\begin{proof}[Proof of Theorem~\ref{Thm.Existence_unipot}]
	We may shrink $X$ such that there exists a Rosenlicht quotient 
	$\pi \colon X \to Y$ for $H$. By shrinking $Y$ 
	(and thus also $X$) we may further assume that there is a Rosenlicht quotient 
	$\eta \colon Y \to Z$ for the closure of the image of $G$ in $\Bir(Y)$. 
	Then $\kk(Z) = \kk(X)^{GH}$.

	Let $K$ be an algebraic closure of $\kk(Z)$ and consider the continuous, injective
	pull-back homomorphism
	\[
		\varepsilon \colon \Bir(X/Z) \to \Bir_K(X_K) \, .
	\]
	Then $\overline{\varepsilon(G)}, \overline{\varepsilon(H)} \subseteq 
	\Bir_K(X_K)$
	are closed, connected, commutative, non-trivial, subgroups such that 
	$\overline{\varepsilon(G)}$ normalizes $\overline{\varepsilon(H)}$. Moreover, 
	by~\eqref{Eq.invariants}
	\[
		K(X_K)^{\overline{\varepsilon(G)} \, \overline{\varepsilon(H)}} 
		= K(X_K)^{\varepsilon(G H)}
		= K \, .
	\]

	Let $A_1 \subseteq A_2 \subseteq A_3 \subseteq \ldots$ be an increasing chain of 
	irreducible closed algebraic subsets of $\Bir(X)$ such that their union is $H$.
	By assumption, there exist countably many $h_1, h_2, \ldots \in H$ such that 
	$H = \bigcup_{i \geq 1} \overline{C_G(h_i)}$, where $C_G(h)$
	is the $G$-conjugacy class of $h$ in $H$. 
	Let $\rho_d \colon V_d \to \Bir(X)$ be a morphism
	with image $A_d$. Then finitely many of the $\rho_d^{-1}(\overline{C_G(h_i)})$, $i \in \NN$
	cover $V_d$ (see e.g.~\cite[Lemma~1.3.1]{FuKr2018On-the-geometry-of}) and since $A_d$
	is irreducible, there exists $i_d \in \NN$ with $A_d \subseteq \overline{C_G(h_{i_d})}$.
	Thus, we get for all $d \in \NN$
	\[
		\varepsilon(A_d) \subseteq \varepsilon \left(\overline{C_G(h_{i_d})} \right) \subseteq
		\overline{\varepsilon(C_G(h_{i_d}))} \subseteq 
		\overline{C_{\overline{\varepsilon(G)}}(\varepsilon(h_{i_d}))}
	\]
	and hence every $\overline{\varepsilon(A_d)}$ is contained in the closure
	of a $\overline{\varepsilon(G)}$-conjugacy class in $\overline{\varepsilon(H)}$.
	As $H$ is covered by countably many $G$-conjugacy class closures, it follows that
	$G$ does not centralize $H$. Thus, $\overline{\varepsilon(G)}$ does not centralize
	$\overline{\varepsilon(H)}$.
	Now, the subgroups $\overline{\varepsilon(G)}, \overline{\varepsilon(H)}$ of 
	$\Bir(X_K)$
	and the chain $\overline{\varepsilon(A_1)} \subseteq 
	\overline{\varepsilon(A_2)} \subseteq \cdots$
	satisfy the assumptions of 
	Lemma~\ref{Lem.Existence_unipot_trivial_inv} and therefore $\overline{\varepsilon(H)}$
	contains unipotent elements different from the identity. 
	
	Let $L$ be an algebraic closure of $\kk(Y)$. Then, the composition
	\[
		\Bir(X/Y) \xrightarrow{\varepsilon |_{\Bir(X/Y)}} \Bir_K(X_K/Y_K)
		\xrightarrow{\delta} \Bir_L(X_L)
	\]
	is equal to the pull-back homomorphism associated to $\pi \colon X \to Y$, 
	where $\delta$ denotes the pull-back homomorphism associated to $\pi_K \colon X_K \to Y_K$. 
	By Proposition~\ref{Prop.Descent_Ga}, $\delta(\overline{\varepsilon(H)}) \subseteq \Bir_L(X_L)$
	contains unipotent elements different from the identity and in particular, 
	$\overline{\delta(\varepsilon(H))}$ as well.
	Hence, Corollary~\ref{Cor.unipot_commuting} yields 
	a unipotent element in $\Bir(X/Y)$ that is different from the identity and 
	commutes with $H$.
\end{proof}

\subsection{The proofs of Theorem~\ref{thm:charmain} and
	Theorem~\ref{thm:charmain2}}
	We start with a result,
	which provides certain self-centralizing closed subgroups of $\Bir(X)$
	(see Proposition~\ref{Prop.coincideswithitscentr}).

	In order to state it  we have to generalize some notion following 
	\cite[\S1]{ReSa2024Maximal-commutativ}. Let  $X$ be an irreducible
	variety endowed with a regular faithful action of an algebraic group $G$.
	For a $G$-invariant rational map $f \colon X \dashrightarrow G$ 
	denote by $\dom(f) \subseteq X$ its domain 
	(which is a $G$-stable open dense subset). Then
	\[
		r_f : \dom(f) \to \dom(f) \, , \quad x \mapsto f(x) \cdot x
	\]
	defines an automorphism of $\dom(f)$ (with inverse $x \mapsto f(x)^{-1} \cdot x$) 
	and hence $r_f$ is an element of 
	$\Bir(X)$. Let $\Rat^G(X, G)$ be the group of $G$-invariant rational maps from
	$X$ to $G$. Then, 
	\[
			\Rat^G(X, G) \to \Bir(X) \, , \quad f \mapsto r_f
	\]
	is a group homomorphism, and we denote its image by $\rRepl_X(G)$.

	If $G$ is just an algebraic subgroup of $\Bir(X)$, we still can define
	$\rRepl_X(G)$ by choosing an 
	irreducible variety $X'$ with a regular faithful $G$-action and a $G$-equivariant
	birational map 
	$\varphi \colon X \dashrightarrow X'$ 
	(Weil regularization, see e.g.~\cite{Br2022On-models-of-algeb}) 
	and then set
	\[
		\rRepl_X(G) \coloneqq 
		\varphi^{-1} \circ \rRepl_{X'}(G) \circ \varphi \subseteq \Bir(X) \, .
	\]
	Note that the definition of $\rRepl_X(G)$ does not depend on the choice of a regular
	model $\varphi \colon X \dashrightarrow X'$ and that the definition of $\rRepl_X(G)$
	slightly differs from that one in~\cite[\S1]{ReSa2024Maximal-commutativ}.

	\begin{remark}
		If $G$ is an algebraic subgroup of $\Bir(X)$, then $\rRepl_{X}(G)$
		lies in $\Bir(X_0/Y)$, where $\pi\colon X_0 \to Y$ is a Rosenlicht
		quotient of $X$ with respect to $G$.
	\end{remark}

	\begin{proposition}
		\label{Prop.coincideswithitscentr}
		Let $G$ be a connected commutative affine algebraic subgroup of $\Bir(X)$. 
		Then, $\rRepl_X(G)$ is self-centralizing in $\Bir(X)$.
	\end{proposition}

	Before we start the proof of Proposition~\ref{Prop.coincideswithitscentr}
	we establish a lemma about regular actions of connected commutative algebraic groups.

	\begin{lemma}
		\label{Lem.Product_model}
		Let $G$ be a connected commutative affine algebraic group that acts faithfully and 
		regularly on $X$. Then there exists an irreducible variety $Y$, 
		a closed connected subgroup
		$G' \subseteq G$ and a $G'$-equivariant open embedding 
		$\varepsilon \colon G' \times Y \to X$ with a $G$-stable image,
		where $G'$ acts via left multiplication only on the first factor of $G' \times Y$,
		and the following composition is a 
		geometric quotient for the $G$-action:
		\[
			\varepsilon(G' \times Y) 
			\xrightarrow[\sim]{\varepsilon^{-1}} G' \times Y \xrightarrow{(g', y) \mapsto y} Y \, .
		\]
	\end{lemma}

	\begin{proof}
		Let $G_u, G_s \subseteq G$ be the closed subgroups
		of unipotent and semisimple elements, respectively. Then $G = G_u G_s$.
		Denote by $X' \subseteq X$ the subset of the $x \in X$ with a trivial stabilizer
		$(G_s)_x$. Then $X'$ is open and dense in $X$ 
		(see e.g.~\cite[Proposition~3.6.1]{Kr2016Algebraic-transfor}) and since $G$
		is commutative, $X'$ is $G$-stable. Hence, by shrinking $X$ we may assume
		that $G_s$ acts freely on $X$. Thus, we get for all $x \in X$
		\[
			G_x = (G_x \cap G_u)(G_x \cap G_s) \subseteq G_u \, .
		\]

		By \cite[Theorem~2]{Br2021Homogeneous-variet} 
		there exists a $G$-stable open dense subset $X_0$ in $X$ such that
		the geometric quotient $\pi \colon X_0 \to Y$ of the $G$-action exists, 
		$X_0$ and $Y$ are smooth and affine,
		$\pi$ is smooth and admits a section $s \colon Y \to X_0$.
		Take $y_0 \in Y$ and let $x_0 = s(y_0) \in X_0$. 
		Let $U \subseteq G_u$ be a closed subgroup such that 
		$U \cap G_{x_0}$ is trivial and $\dim U + \dim G_{x_0} = \dim G_u$.
		Denote $G' = U G_s \subseteq G$ and consider the $G'$-equivariant morphism
		\[
			\varepsilon \colon G' \times Y \to X_0 \, , \quad 
			(g', y) \mapsto g' \cdot s(y) \, .
		\]

		\begin{claim}
			After shrinking $Y$, the morphism $\varepsilon$ is bijective.
		\end{claim}

		\begin{proof}
			Let $E \subseteq U \times Y$ be the closed subset of 
			those $(u, y) \in U \times Y$ with $u s(y) = s(y)$.
			Consider the surjective morphism $p \colon E \to Y$, $(u, y) \mapsto y$
			and let $E' \subseteq E$ be the irreducible
			component of $E$ that contains $\{1\} \times Y$. 
			Since $p^{-1}(y_0) = \{ (1, y_0) \}$, by the semi-continuity
			of the fibre dimension (see e.g.~\cite[Theorem~14.112]{GoWe0Algebraic-geometryI})
			applied to $p \colon E \to Y$, 
			there is a dense open subset $E'' \subseteq E'$ that contains $(1, y_0)$ and 
			is also open in $E$ with 
			\[
				\dim_{e''} p^{-1}(p(e'')) = 0 \quad \textrm{for all $e'' \in E''$} \, .
			\]
			Since for all $y \in Y$ the fibre of $p \colon E \to Y$ over $y$ 
			is irreducible (it is the stabilizer $U_{s(y)}$) 
			and contains $(1, y)$, we conclude that $p^{-1}(y) = \{(1, y)\}$ for 
			all $y \in p(E'') \subseteq Y$.
			Hence, $p |_{E''} \colon E'' = p^{-1}(p(E'')) \to Y$ is an open embedding 
			(see \cite[Corollary~12.88]{GoWe0Algebraic-geometryI}).
			This shows that the stabilizer $(G')_{s(y)} = U_{s(y)}$ is trivial for $y \in p(E'')$. 
			Thus, after replacing $Y$ with $p(E'')$,
			we may assume that $\varepsilon$ is injective (and still $y_0 \in Y$). 
			By the smoothness of $\pi$,
			it follows that all $G$-orbits in $X_0$ have the same dimension.
			As $\dim G x_0 = \dim G - \dim G_{x_0} = \dim G_s + \dim U = \dim G'$
			and since for all $y \in Y$ 
			the $G'$-orbits in $G s(y)$ have the same dimension,
			we conclude that $\varepsilon$ is also surjective.
		\end{proof}

		Now, the statement follows from Zariski's main theorem.
	\end{proof}

	As $G$ is commutative in Proposition~\ref{Prop.coincideswithitscentr}, 
	$\rRepl_X(G)$ is commutative. So, for establishing a 
	proof of Proposition~\ref{Prop.coincideswithitscentr}, we only have to show that every 
	element of $\Bir(X)$ that commutes with $\rRepl_X(G)$ is contained in $\rRepl_X(G)$.

	\begin{proof}[Proof of Proposition~\ref{Prop.coincideswithitscentr}]
		By replacing $X$ with a regular model for $G$, we may assume that $G$
		acts faithfully and regularly on $X$. 
		Using Lemma~\ref{Lem.Product_model} (and after possibly shrinking $X$) 
		we may assume that
		there exists a $G'$-equivariant isomorphism
		\[
			\varepsilon \colon X' \coloneqq G' \times Y \xrightarrow{\sim} X
		\]
		for a connected closed subgroup $G'$ of $G$ and 
		$\pr_Y \circ \varepsilon^{-1}$ is a geometric quotient for the $G$-action,
		where $\pr_Y$ is the projection to $Y$. In particular, 
		$\rRepl_{X'}(G') \subseteq \varepsilon^{-1} \circ \rRepl_X(G) \circ \varepsilon$.
		Let $\varphi \in \Cent_{\Bir(X)}(\rRepl_{X}(G))$.
		Then,
		$\psi \coloneqq \varepsilon^{-1} \circ \varphi \circ \varepsilon$ commutes with
		$\rRepl_{X'}(G')$. Moreover, $\lociso(\psi)$ is $G'$-stable and hence, 
		we may assume after shrinking $Y$ (and then $X$ and $X'$) 
		that $\psi$ is an automorphism of $X'$.

		\begin{claim}
			\label{Claim.Id_on_quotient}
			$\psi$ induces the identity on $Y$.
		\end{claim}

		\begin{proof}
			Assume this is not the case. Then there exist $g_0, g_1$ in $G'$ and $y_0 \neq y_1$ in $Y$
			such that $\psi(g_0, y_0) = (g_1, y_1)$. Choose a
			$G'$-invariant rational map $f \colon G' \times Y \dashrightarrow G'$
			with $(g_0, y_0), (g_1, y_1) \in \dom(f)$ and $f(g_0, y_0) \neq f(g_1, y_1)$. Then:
			\begin{eqnarray*}
				f(g_0, y_0) \cdot (g_1, y_1) = 
				\psi(f(g_0, y_0) \cdot (g_0, y_0)) &=& \psi(r_f(g_0, y_0)) \\
				&=& r_f(\psi(g_0, y_0))= f(g_1, y_1) \cdot (g_1, y_1) \, .
			\end{eqnarray*}
			As $G'$ acts freely on $G' \times Y$, we get $f(g_0, y_0) = f(g_1, y_1)$, contradiction.
		\end{proof}

		Let $h \colon G' \times Y \to G'$ be defined via $h(g', y) = \pr_{G'}(\psi(1, y))$, 
		where $\pr_{G'}$ denotes the projection to $G'$.
		Then we get for all $v = (g', y) \in G' \times Y$
		\[
			\psi(v) = \psi(g' \cdot (1, y)) = g' \cdot \psi(1, y) = (g' h(g', y), y)
			= h(v) \cdot v \, ,
		\]
		where the third equality uses Claim~\ref{Claim.Id_on_quotient}.
		This implies the proposition.
	\end{proof}

	We call an element $g$ in a group $G$ \emph{divisible by $k$} if there exists
	an element $f \in G$ such that $f^k = g$. An element is called \emph{divisible} 
	if it is divisible by all
	$k > 0$. If $G$ is an algebraic group, then e.g.~by \cite[Lemma 3.12]{LiReUr2023Characterization-o} for any $g \in G$ there exists
	$k > 0$ that depends on $g$ such that $g^k$ is a divisible element in $G$.
	
	\begin{proposition}\label{Prop.notdivisible}
	Assume $\dim X > 0$. 
	Then there is an element  $g \in \kk(X)^\ast$ such that
	no power of $g$ is divisible in $\kk(X)^\ast$. 
	\end{proposition}
	\begin{proof}
	Note that $\kk(X)$ is a finite extension of the field $\kk(t_1,\dots,t_n)$,
	where $t_1,\dots,t_n$ are algebraically independent in $\kk(X)$ and $n = \dim X > 0$.
	We claim that no power of  $t_1$ is divisible in $\kk(X)^\ast$. Let $l \geq 1$ be a fixed integer.
	First note that for no prime $p > l$, the polynomial
	\[
		T^p - t_1^l \in \kk(t_1, \ldotp, t_n)[T]
	\]
	has a root in $\kk(t_1,\dots,t_n)$, and thus it
	is irreducible over $\kk(t_1, \ldotp, t_n)$,
	see e.g.~\cite[Ch. VI, Theorem~6.2(ii)]{La2002Algebra}. 
	If $p > l$ is a prime that does not divide the degree of the field extension
	$\kk(t_1, \ldots, t_n) \subseteq \kk(X)$, then $T^p-t_1^l$ has no root in 
	$\kk(X)$, and thus $t_1^l$ is not divisible by $p$ in $\kk(X)^\ast$. 
	\end{proof}
	
	Denote by $\Tr_n \subset \Bir(\mathbb{A}^n)$ the subgroup of translations,~i.e.,
	\[
		\Tr_n = 
		\set{ (x_1,\dots,x_n) \mapsto (x_1 + c_1,\dots,x_n + c_n) }
	{c_i \in \kk} \subseteq \Aut(\AA^n) \subseteq 
	\Bir(\AA^n) \, ,
	\]
	by $\T_n$ the standard torus in $\Aut(\AA^n)$,~i.e.,
	\[
		\T_n = \set{(x_1,\dots,x_n) \mapsto (a_1 x_1,\dots, a_n x_n)}
		{a_i \in \kk^\ast} \subseteq \Aut(\AA^n) \subseteq \Bir(\AA^n)
	\]
	and by $\Aff_n$ the subgroup of affine transformations,~i.e.
	the subgroup generated by all linear transformations of $\AA^n$
	and the translations $\Tr_n$.

	\begin{lemma}\label{Lem.normalizeroftranslationsincremona}
	The normalizer $\Norm_{\Bir(\mathbb{A}^n)}(\Tr_n)$  of $\Tr_n$ in $\Bir(\mathbb{A}^n)$ coincides with $\Aff_n$.
	\end{lemma}
	\begin{proof}
 	Let $\varphi \in \Norm_{\Bir(\mathbb{A}^n)}(\Tr_n)$.
	Then $\lociso(\varphi)$ intersects $\AA^n$, say in $x_0$. For any $x_1 \in \AA^n$,
	there exists a translation $t \in \Tr_n$ such that $x_0 = t(x_1)$. As
	$\varphi$ normalizes $\Tr_n$ there exists $t' \in \Tr_n$ with 
	$t' \circ \varphi \circ t = \varphi$. Using  that $t' \circ \varphi \circ t$ is a local
	isomorphism at $x_1$, we get thus that $\varphi$ is a local isomorphism at $x_1$ as well.
	So, $\varphi \in \Aut(\mathbb{A}^n)$. 
	
	Let $\varphi_1, \ldots, \varphi_n$ be the coordinate functions of the automorphism 
	$\varphi$ of $\AA^n$.
	Using again that $\varphi$ normalizes $\Tr_n$, we get for all $c \in \kk^n$
	and all $i = 1, \ldots, n$ that
	\[
		\varphi_i(x+c)- \varphi_i(x) \in \kk
	\]
	is a constant polynomial. This implies that
	for all $c \in \kk^n$ and all $j = 1, \ldots, n$, we get for the partial derivative 
	\[
		\frac{\partial \varphi_i}{\partial x_j}(x+c) - 
		\frac{\partial \varphi_i}{\partial x_j}(x) = 0 \, .
	\]
	Hence, 
	$\frac{\partial \varphi_i}{\partial x_j} \in \kk$ and thus $\deg \varphi_i \le 1$ 
	for all $i=1, \ldots, n$, i.e.~$\varphi \in \Aff_n$.
	\end{proof}
	
	The next lemma is an analog of \cite[Lemma 5.1]{ReSa2024Maximal-commutativ} and constitutes
	our last ingredient in the proof of Theorem~\ref{thm:charmain}.
	
	\begin{lemma}\label{lemmamaximalunipotent} 
		Let 
		$G$ be a closed commutative  subgroup of $\Bir(X)$ which consists of unipotent elements. 
		Then there exists a commutative unipotent algebraic subgroup $U \subseteq \Bir(X)$ contained
		in $G$ such that $G \subseteq \rRepl_X(U)$ and $\kk(X)^U = \kk(X)^G$.
	\end{lemma}
	\begin{proof} 
	We establish first the following claim:

	\begin{claim}
		There is a unipotent algebraic subgroup $U \subseteq G$ such that $\kk(X)^U = \kk(X)^G$.
	\end{claim}

	\begin{proof}
	Let $U_1 \subseteq G$ be a one-dimensional unipotent subgroup. If $\kk(X)^{U_1} = \kk(X)^G$, then we set $U=U_1$ and the claim is proven. Hence, we can assume that $\kk(X)^G \subsetneq \kk(X)^{U_1}$. 
	Pick an element $u_2 \in G \setminus U_1$ such that $(\kk(X)^{U_1})^{u_2} \subsetneq \kk(X)^{U_1}$.
	Hence, $U_2 \coloneqq U_1 \circ \overline{\langle u_2 \rangle}$ is an algebraic subgroup of
	$\Bir(X)$ contained in $G$, isomorphic to $\GG_a^2$ and $\kk(X)^{U_2} \subsetneq \kk(X)^{U_1}$.
	Again, if $\kk(X)^{U_2} = \kk(X)^G$, then we set $U= U_2$ and the claim is proven. Hence, we can assume that $\kk(X)^G \subsetneq \kk(X)^{U_2}$.
	We continue in this way and
	build a successive chain of algebraic subgroups $U_1 \subsetneq U_2 \subsetneq
	\cdots$ inside $G$ such that $\dim U_i = i$ and $\kk(X)^{U_i} \subsetneq (\kk(X)^{U_{i-1}}$ for all $i \ge 1$.
	By e.g.~Theorem~\ref{Thm.Rosenlicht} 
	we have that the transcendence degree  of $\kk(X)^{U_{i-1}}$ is one more than
	that of $\kk(X)^{U_i}$. Hence, the sequence $U_1 \subsetneq U_2 \subsetneq \dots$ stops with some 
	$U = U_k$, $k \ge 1$ and
	$\kk(X)^U = \kk(X)^G$. This proves the claim. 
	\end{proof}
	
	Since any element of $\rRepl_X(U)$ commutes with $G$ 
	(here we use that $\kk(X)^G = \kk(X)^U$) we get that $G$ is contained in the centralizer
	of $\rRepl_X(U)$ in $\Bir(X)$. The statement follows, since $\rRepl_X(U)$
	is self-centralizing, see Proposition~\ref{Prop.coincideswithitscentr}.
\end{proof}

\begin{proof}[Proof of Theorem~\ref{thm:charmain}]
	Let $\Xi\colon \Bir(\mathbb{A}^n) \to \Bir(X)$ be an isomorphism of groups. Note that 
	the translations $\Tr_n$ and the torus $\T_n$ both act 
	with a dense open orbit on $\AA^n$ and thus
	Corollary~\ref{Cor.Centralizer} implies that $\Tr_n$, $\T_n$ are both self-centralizing.
	In particular, the images $\Xi(\Tr_n), \Xi(\T_n)$ are closed subgroups
	of $\Bir(X)$. Since every countable index subgroup $A$ of $\Tr_n$ is dense in $\Tr_n$,
	we get $\Cent_{\Bir(\AA^n)}(A) = \Cent_{\Bir(\AA^n)}(\Tr_n) = \Tr_n$ and in particular,
	\begin{equation}
		\label{Eq.Centralizer_image_Tr_n_circ}
		\Cent_{\Bir(X)}(\Xi(\T_n)^\circ) = \Xi(\T_n) \, .
	\end{equation}
	Since $\T_n$ acts by conjugation with finitely many orbits on $\Tr_n$, it 
	follows that $\Xi(\T_n)^\circ$ acts with countably many orbits on $\Xi(\Tr_n)^\circ$.
	By Theorem~\ref{Thm.Existence_unipot}
	applied to $H = \Xi(\Tr_n)^\circ$ and $G = \Xi(\T_n)^\circ$
	there exist unipotent elements in $\Bir(X)$
	that commute with $\Xi(\Tr_n)^\circ$ and hence, $\Xi(\Tr_n)$ contains unipotent elements
	by~\eqref{Eq.Centralizer_image_Tr_n_circ}.
	As $\SL_n \subseteq \Bir(\mathbb{A}^n)$ acts transitively on 
	$\Tr_n \setminus \{ \id_{\AA^n} \}$ by conjugation, we conclude that  
	all elements of $\Xi(\Tr_n) \setminus \{ \id_X \}$ are conjugate in $\Bir(X)$. 
	Hence, all elements of $\Xi(\Tr_n)$ are unipotent. In particular,
	$\Xi(\Tr_n)$ is connected.

	Assume first $\kk(X)^{\Xi(\Tr_n)} = \kk$. Then by Corollary~\ref{Cor.Centralizer},
	it is a commutative, unipotent algebraic subgroup of $\Bir(X)$ of dimension $k \coloneqq \dim X$.
	Using the Weil regularization theorem (see e.g.~\cite{Br2022On-models-of-algeb}), we may assume that $\Xi(\Tr_n)$ acts regularly on $X$.
	Hence, the dense open orbit of $X$ is isomorphic to $\Xi(\Tr_n)$. Again after conjugation,
	we may assume that $X = \AA^k$ and $\Xi(\Tr_n) = \Tr_k$.
	By Lemma \ref{Lem.normalizeroftranslationsincremona} the normalizer of 
	$\Tr_n \subset \Bir(\mathbb{A}^n)$ is isomorphic to $\Aff_n$
	and the normalizer of $\Tr_k \subseteq \Bir(\AA^k)$ is isomorphic to $\Aff_k$. 
	Therefore, $\Xi$ induces an isomorphism of the groups $\Aff_n$ and $\Aff_k$ 
	which sends the translations $\Tr_n \subseteq \Aff_n$ to the translations 
	$\Tr_k \subseteq \Aff_k$. Therefore, $\Xi$ induces an isomorphism from 
	$\Aff_n/\Tr_n \simeq \GL_n$ to 
	$\Aff_k/\Tr_k \simeq \GL_k$. 
	Then the commutators $\SL_n = [\GL_n, \GL_n]$ and $\SL_k = [\GL_k, \GL_k]$ are isomorphic and hence,
	the centers of them are isomorphic. But this says $n = k$.

	Assume now that 
	$\kk(X)^{\Xi(\Tr_n)} \neq \kk$. 
	Let $U \subseteq \Xi(\Tr_n)$ be an algebraic subgroup with 
	\begin{equation}
		\label{Eq.Nice_Uipotent_U}
		\Xi(\Tr_n) \subseteq \rRepl_X(U)
		\quad \textrm{and} \quad \kk(X)^U = \kk(X)^{\Xi(\Tr_n)} \, ,
	\end{equation}
	see Lemma~\ref{lemmamaximalunipotent}. 
	Using Weil's regularization theorem again, we may assume that $U$ acts regularly and faithfully
	on $X$. By Lemma~\ref{Lem.Product_model} there exists a closed subgroup $U' \subseteq U$
	and a birational $U'$-equivariant morphism $\varepsilon \colon U' \times Y \dashrightarrow X$
	for some variety $Y$ with $\varepsilon^\ast(\kk(X)^U) = \kk(Y)$.
	As 
	$\rRepl_{U' \times Y}(U') \subseteq \varepsilon^{-1} \circ \rRepl_X(U) \circ \varepsilon$,
	as $\rRepl_{U' \times Y}(U')$ is self-centralizing 
 	in $\Bir(U' \times Y)$ (see Proposition~\ref{Prop.coincideswithitscentr}) and since
	$\rRepl_X(U)$ is commutative, it follows that
	\[
		\rRepl_{U' \times Y}(U') = \varepsilon^{-1} \circ \rRepl_X(U) \circ \varepsilon \, .
	\]
	Moreover, $\varepsilon^\ast(\kk(X)^{\Xi(\Tr_n)}) = 
	\varepsilon^\ast(\kk(X)^U) = \kk(Y) = \kk(U' \times Y)^{U'}$.
	Thus, we may replace $X$ with $U' \times Y$ and $U$ with $U'$, and~\eqref{Eq.Nice_Uipotent_U}
	is still satisfied.

	As $\rRepl_X(U)$ is commutative and $\Tr_n$
	is self-centralizing in $\Bir(\AA^n)$, we get from~\eqref{Eq.Nice_Uipotent_U}
	\[
		\Xi^{-1}(\rRepl_X(U)) = \Tr_n \, .
	\]
	Denote by $D$ the standard subtorus of $\Bir(U) \subset \Bir(U \times Y) = \Bir(X)$. 
	Then $\kk(X)^D = \kk(Y) = \kk(X)^U = \kk(X)^{\Xi(\Tr_n)} \neq \kk$.
	The subgroup $\rRepl_X(D) \subset \Bir(X)$ is self-cen\-tra\-li\-zing by Proposition \ref{Prop.coincideswithitscentr} and normalizes $\rRepl_X(U)$. Hence, 
	\[
		A \coloneqq \Xi^{-1}(\rRepl_X(D))
	\] 
	is a closed subgroup of $\Bir(\AA^n)$ which normalizes $\Tr_n$. 
	By Lemma~\ref{Lem.normalizeroftranslationsincremona}, $A$ is contained in $\Aff_n$.
	Hence, for every element $a \in A$
	there exists  $k > 0$ such that $a^k$ is
	divisible in $A$, see~e.g.~\cite[Lemma 3.12]{LiReUr2023Characterization-o}. 

	Note that for a general $x \in X$ the stabilizer $D_x$
	is trivial and thus $\Rat^D(X, D) \to \rRepl_X(D)$, $f \mapsto r_f$
	is bijective. Hence, we get the following group isomorphism
	\[
		\rRepl_{X}(D) \simeq \Rat^D(X, D) \simeq \Rat(Y, D) \simeq
		\kk(Y)^\ast \times \cdots \times \kk(Y)^\ast \, .
	\]
	However, using Lemma \ref{Prop.notdivisible} and $\kk(Y) \neq \kk$, there exist
	elements in $\kk(Y)^\ast$ such that no power of it is divisible in $\kk(Y)^\ast$.
	This contradiction proves the theorem.
\end{proof}

\begin{proof}[Proof of Theorem~\ref{thm:charmain2}]
	Let $\Xi \colon \Bir(\AA^1 \times Y) \to \Bir(X)$ be a group isomorphism. Denote by 
	$U \subseteq \Bir(\AA^1 \times Y)$ 
	the algebraic subgroup of automorphisms $(t, y) \mapsto (t + \lambda, y)$
	for $\lambda \in \GG_a$ and by $D \subseteq \Bir(\AA^1 \times Y)$ the algebraic subgroup
	of automorphisms $(t, y) \mapsto (\mu t, y)$ for $\mu \in \GG_m$. Then 
	$\rRepl(D)$ acts via conjugation with two orbits on 
	$\rRepl(U)$ (this action is given by
	$\kk(Y)^{\ast} \times \kk(Y) \to \kk(Y)$, $(p, q) \mapsto pq$). 
	Hence, the same holds for the images of $\rRepl(D)$
	and $\rRepl(U)$ in $\Bir(X)$. As $\rRepl(D)$ and $\rRepl(U)$ are self-centralizing
	in $\Bir(\AA^1 \times Y)$ (see Proposition~\ref{Prop.coincideswithitscentr}) 
	the same holds for their images in $\Bir(X)$ and in particular, they are closed in $\Bir(X)$.
	Hence, $G \coloneqq \Xi(\rRepl(D))^\circ$ and $H \coloneqq \Xi(\rRepl(U))^\circ$ are closed connected commutative non-trivial
	subgroups of $\Bir(X)$ and $G$ acts via conjugation on $H$ with countably many orbits. By Theorem~\ref{Thm.Existence_unipot} there exists a unipotent element in $\Bir(X)$
	that is different from the identity. By \cite[Theorem~1]{Ma1963On-algebraic-group} we conclude
	that $X$ is ruled.
\end{proof}

\par\bigskip
\renewcommand{\MR}[1]{}
\bibliographystyle{amsalpha}
\bibliography{newBIB}

\end{document}